\documentclass[12pt,reqno]{amsart}

\title{Entropic Projections and Dominating Points}
\author{Christian L\'eonard}
\date{March 09 - Revised version (round 2)}

\usepackage{amssymb, amsmath, amsfonts, latexsym, enumerate}
\usepackage[usenames,dvipsnames]{pstricks}
\usepackage{epsfig}

\newtheorem{theorem}[equation]{Theorem}
\newtheorem{lemma}[equation]{Lemma}
\newtheorem{proposition}[equation]{Proposition}
\newtheorem{corollary}[equation]{Corollary}

\newtheorem{definition}[equation]{Definition}
\newtheorem{definitions}[equation]{Definitions}

\theoremstyle{remark}
\newtheorem{remark}[equation]{Remark}
\newtheorem{remarks}[equation]{Remarks}
\newtheorem{example}[equation]{Example}

\numberwithin{equation}{section}

\newcommand{\R}{\mathbb{R}}

\newcommand{\1}{\textbf{1}}

\newcommand{\dom}{\mathrm{dom\,}}
\newcommand{\icordom}{\mathrm{icordom\,}}
\newcommand{\icor}{\mathrm{icor\,}}
\newcommand{\cl}{\mathrm{cl\,}}

\newcommand{\inter}{\mathrm{int\,}}

\newcommand{\lsc}{lower semicontinuous}
\newcommand{\usc}{upper semicontinuous}

\newcommand{\boulette}[1]{$\bullet$\ Proof of #1.}
\newcommand{\Boulette}[1]{\par\medskip\noindent $\bullet$\ Proof of #1.}

\newcommand\seq[2]{(#1_#2)_{#2\ge1}}
\newcommand\Lim[1]{\lim_{#1\rightarrow\infty}}
\newcommand\Liminf[1]{\liminf_{#1\rightarrow\infty}}
\newcommand\Limsup[1]{\limsup_{#1\rightarrow\infty}}

\newcommand{\limn}{\lim_{n\rightarrow\infty}}

\newcommand{\Ib}{\bar I}
\newcommand{\Jb}{\bar J}
\newcommand{\Il}{I_\lambda}
\newcommand{\Ils}{I_{\lambda^*}}

\newcommand{\Ilsb}{\Ib_{\lambda^*}}

\newcommand{\YY}{\mathcal{Y}}
\newcommand{\YYo}{\YY_o}
\newcommand{\YYt}{\widetilde{\YY}}
\newcommand{\YYb}{\overline{\YY}}
\newcommand{\XX}{\mathcal{X}}
\newcommand{\XXo}{\XX_o}
\newcommand{\ZZ}{\mathcal{Z}}
\newcommand{\CC}{\mathcal{C}}
\newcommand{\CCb}{\overline{\CC}}
\newcommand{\CX}{C\cap\XX}
\newcommand{\UUo}{\mathcal{U}_o}
\newcommand{\LLo}{\mathcal{L}_o}

\newcommand{\Yy}{\YYo}
\newcommand{\YE}{\YY}
\newcommand{\YL}{\YY}
\newcommand{\Xx}{\XXo}
\newcommand{\XE}{\XX}
\newcommand{\XL}{\XX}

\newcommand{\MZ}{M_\ZZ}
\newcommand{\PZ}{P_\ZZ}
\newcommand{\OZ}{O_{\mathrm{exp}}}
\newcommand{\Eexp}{\mathcal{E}_{\mathrm{exp}}}
\newcommand{\Lexp}{\mathcal{L}_{\mathrm{exp}}}

\newcommand{\lmax}{{\lambda_{\diamond}}}
\newcommand{\lmaxs}{{\lambda_{\diamond}^*}}
\newcommand{\El}{E_\lmax}
\newcommand{\Ll}{L_\lmax}
\newcommand{\Lls}{L_\lmaxs}
\newcommand{\LlsR}{L_\lmaxs R}

\newcommand{\IZ}{\int_{\ZZ}}

\newcommand{\pc}{\mathrm{P}_C}
\newcommand{\pbc}{\overline{\mathrm{P}}_C}
\newcommand{\pcc}{\mathcal{P}_\CC}
\newcommand{\dc}{\mathrm{D}_C}
\newcommand{\dbc}{\overline{\mathrm{D}}_C}

\newcommand{\dtc}{\widetilde{\mathrm{D}}_C}
\newcommand{\pE}{\mathrm{P}_E}
\newcommand{\PE}{($\pE$)}
\newcommand{\pL}{\mathrm{P}_L}
\newcommand{\PL}{($\pL$)}
\newcommand{\dE}{\mathrm{D}_E}

\newcommand{\dL}{\mathrm{D}_L}

\newcommand{\sLE}{\sigma(\LlsR,\El)}
\newcommand{\sLL}{\sigma(\LlsR,\Ll)}
\newcommand{\NF}{\|\cdot\|_{\lmax}}

\newcommand{\lb}{\bar\ell}
\newcommand{\lh}{\hat\ell}
\newcommand{\xb}{\bar x}
\newcommand{\xh}{\hat x}
\newcommand{\ob}{\bar \omega}
\newcommand{\ot}{\tilde \omega}
\newcommand{\od}{\omega_{\diamond}}

\newcommand{\moyn}{\frac 1n\sum_{i=1}^n}

\newcommand{\Qh}{\widehat{Q}}
\newcommand{\Qd}{Q_\diamond}

\newcommand{\ul}{\langle u,\ell\rangle}

\newcommand{\yx}{\langle y,x\rangle}
\newcommand{\yt}{\langle y,\theta\rangle}
\newcommand{\ott}{\langle \ot,\theta\rangle}
\newcommand{\tl}{\langle\theta,\ell\rangle}

\newcommand{\limsupdn}{\limsup_{\delta\rightarrow 0}\Limsup n}
\newcommand{\liminfdn}{\liminf_{\delta\rightarrow 0}\Liminf n}

\newcommand{\AB}{{A\!\times\! B}}
\newcommand{\IAB}{\int_{\AB}}

\begin{document}


 \address{Modal-X, Universit\'e Paris Ouest. B\^at.\! G,
 200 av. de la R\'epublique. 92000 Nanterre, France}
 \email{christian.leonard@u-paris10.fr}

\begin{abstract}
Entropic projections and dominating points are solutions to convex
minimization problems related to conditional laws of large
numbers. They appear in many areas of applied mathematics such as
statistical physics, information theory, mathematical statistics,
 ill-posed inverse problems or large deviation theory. By means of convex conjugate
duality and functional analysis, criteria are derived for the
existence of entropic projections, generalized entropic
projections and dominating points. Representations of the
generalized entropic projections are obtained. It is shown that
they are the ``measure component" of the solutions to some
extended entropy minimization problem. This approach leads to new
results and offers a unifying point of view. It also  permits to
extend previous results on the subject by removing unnecessary
topological restrictions. As a by-product, new proofs of already
known results are provided.
\end{abstract}

\maketitle \tableofcontents


\section{Introduction}
Entropic projections and dominating points are solutions to convex
minimization problems related to conditional laws of large
numbers. They appear in many areas of applied mathematics such as
statistical physics, information theory, mathematical statistics,
 ill-posed inverse problems or large deviation theory.

\subsection*{Conditional laws of large numbers}
Suppose that the empirical measures
\begin{equation}\label{eq-26}
    L_n:=\moyn \delta_{Z_i},\quad n\ge1,
    \index{Ln@$L_n$}
\end{equation}
of the $\ZZ$-valued random variables $Z_1,Z_2,\dots$ ($\delta_z$
is the Dirac measure at $z$) obey a Large Deviation Principle
(LDP) in the set $\PZ\index{PZ@$\PZ,$ set of probability measures
on $\ZZ$}$ of all probability measures on $\ZZ$ with the rate
function $I.$ This approximately means that $\mathbb{P}(L_n\in
\mathcal{A})\underset{n\rightarrow \infty}{\asymp}\exp
[-n\inf_{P\in \mathcal{A}}I(P)]$ for $\mathcal{A}\subset\PZ.$ With
regular enough subsets $\mathcal{A}$ and $\CC\index{C1@$\CC,$
constraint set}$ of $\PZ,$ we can expect that for ``all''
$\mathcal{A}$
\begin{equation}\label{eq-02}
    \Lim n\mathbb{P}(L_n\in
\mathcal{A}\mid L_n\in
\CC)=\left\{%
\begin{array}{ll}
    1, & \hbox{if } \mathcal{A}\ni P_* \\
    0, & \hbox{otherwise} \\
\end{array}%
\right.
\end{equation}
where $P_*$ is a minimizer of $I$ on $\CC.$ To see this, remark
that (formally) $\mathbb{P}(L_n\in \mathcal{A}\mid L_n\in
\CC)\underset{n\rightarrow \infty}{\asymp}\exp[-n(\inf_{P\in
\mathcal{A}\cap \CC}I(P)-\inf_{P\in \CC}I(P))].$ A rigorous
statement is given at Theorem \ref{res-09}.
\\
If  $(Z_1,\dots,Z_n)$ is exchangeable for each $n$, then
\eqref{eq-02} is equivalent to
\begin{equation*}
    \mathcal{L}(Z_1\mid
    L_n\in\CC)\underset{n\rightarrow\infty}{\longrightarrow}P_*
\end{equation*}
which means that conditionally on $L_n\in\CC,$ the law of any
tagged ``particle" (here we chose the first one) tends to $P_*$ as
$n$ tends to infinity.

If $I$ is strictly convex and $\CC$ is convex, $P_*$ is unique and
\eqref{eq-02} roughly means that conditionally on $L_n\in \CC,$ as
$n$ tends to infinity $L_n$ tends to the solution $P_*$ of the
minimization problem
\begin{equation}\label{eq-23}
      \textsl{minimize } I(P)
    \textsl{ subject to } P\in\CC,\quad P\in\PZ
\end{equation}
Such conditional Laws of Large Numbers (LLN) appear in information
theory and in statistical physics where they are often called
Gibbs conditioning principles (see \cite[Section 7.3]{DZ} and the
references therein). If the variables $Z_i$ are independent and
identically distributed with law $R,$ the LDP for the empirical
measures is given by Sanov's theorem and the rate function $I$ is
the relative entropy
    $$I(P|R)=\IZ\log\left(\frac{dP}{dR}\right)\,dP, \quad P\in\PZ.
    \index{I6@$I(\cdot\mid R),$ relative entropy, see \eqref{eq-54}}
    $$
Instead of the empirical probability measure of a random sample,
one can consider another kind of random measure. Let
$z_1,z_2,\dots$ be deterministic points in $\ZZ$ such that the
empirical measure $\moyn \delta_{z_i}$ converges to
$R\in\PZ.$\index{R@$R,$ reference measure} Let $ W_1,W_2,\dots$ be
a sequence of \emph{independent} random real variables. The random
measure of interest is
\begin{equation}\label{eq-27}
    L_n=\moyn W_i\delta_{z_i}
    \index{Ln@$L_n$}
\end{equation}
where the $W_i$'s are interpreted as random weights. If the
weights are independent copies of $W,$ as $n$ tends to infinity,
$L_n$ tends to the deterministic measure $\mathbb{E}W.R$ and obeys
the LDP in the space $\MZ\index{MZ@$\MZ,$ space of signed measures
on $\ZZ$}$ of measures on $\ZZ$ with rate function
$I(Q)=\IZ\gamma^*(\frac{dQ}{dR})\,dR,$ $Q\in\MZ$ where $\gamma^*$
is the Cram\'er transform of the law of $W.$ In case the $W_i$'s
are not identically distributed, but have a law which depends
(continuously) on $z_i,$ one can again show that under additional
assumptions $L_n$ obeys the LDP in $\MZ$ with rate function
\begin{equation}\label{eq-25}
    I(Q)=\left\{%
\begin{array}{ll}
    \IZ \gamma^*_z(\frac{dQ}{dR}(z))\,R(dz), & \hbox{if }Q\prec R \\
    +\infty, & \hbox{otherwise} \\
\end{array}%
\right.    , \quad Q\in\MZ
\end{equation}
\index{I3@$I,$ entropy, see \eqref{eq-16}, \eqref{eq-08}} where
$\gamma^*_z$
    \index{Functions@Functions of $(s,z)$ or $(t,z)$!gammastar@$\gamma^*,$ integrand, see \eqref{eq-25}}
is the Cram\'er transform of $W_z.$ As $\gamma^*$ is the convex
conjugate of the log-Laplace transform of $W,$ it is a convex
function: $I$ is a convex integral functional which is often
called an \emph{entropy}. Again, conditional LLNs hold for $L_n$
and lead to the entropy minimization problem:
\begin{equation}\label{eq-24}
      \textsl{minimize } I(Q)
    \textsl{ subject to } Q\in\CC,\quad Q\in\MZ
\end{equation}
The large deviations of these random measures and their
conditional LLNs enter the framework of Maximum Entropy in the
Mean (MEM) which has been studied among others by
Dacunha-Castelle, Csisz\'ar, Gamboa, Gassiat, Najim see
\cite{CGG,DDCG90,GG97,Naj02} and also \cite[Theorem 7.2.3]{DZ}.
This problem also arises in the context of statistical physics. It
has been studied among others by Boucher, Ellis, Gough, Puli and
Turkington, see \cite{BET99,EGP93}.

The relative entropy corresponds to $\gamma^*(t)=t\log t-t+1$ in
\eqref{eq-25}: the problem \eqref{eq-23} is a special case of
\eqref{eq-24} with the additional constraint that $Q(\ZZ)=1.$

In this paper, the constraint set $\CC$ is assumed to be
\emph{convex} as is almost always done in the literature on the
subject. This allows to rely on convex analysis, saddle-point
theory or on the geometric theory of projection on convex sets.

The convex indicator of a subset $A$ of $X$ is denoted by
\begin{equation}\label{eq-42}
    \iota_A(x)=\left\{%
\begin{array}{ll}
    0, & \hbox{if }x\in A \\
    +\infty, & \hbox{otherwise} \\
\end{array}%
\right.,\quad x\in X. \index{I1@$\iota_A,$ convex indicator, see
\eqref{eq-42}}
\end{equation}
Note that although $\iota_A$ is a convex function if and only if
$A$ is a convex set, we call it a \emph{convex} indicator in any
case to differentiate it from the probabilistic indicator
$\mathbf{1}_A.$
\subsection*{Examples}
We give a short list of popular entropies. They are described at
\eqref{eq-25} and characterized by their integrand $\gamma^*.$
\begin{enumerate}\index{Functions@Functions of $(s,z)$ or $(t,z)$!gammastar@$\gamma^*,$
integrand, see \eqref{eq-25}}
    \item The relative entropy: $\gamma^*(t)=t\log t-t+1+\iota_{\{t\ge0\}}.$
    \item The reverse relative entropy: $\gamma^*(t)=t-\log t-1+\iota_{\{t>0\}}.$
    \item The Fermi-Dirac entropy: $\gamma^*(t)=\frac 12[(1+t)\log(1+t)+(1-t)\log(1-t)]+\iota_{\{-1\le t\le 1\}}.$
    \item  The $L_p$ norm ($1<p<\infty$): $\gamma^*(t)=|t|^p/p$ and
    \item the $L_p$ entropy ($1<p<\infty$): $\gamma^*(t)=t^p/p+\iota_{\{t\ge0\}}.$
\end{enumerate}
Note that the reverse relative entropy is $P\mapsto
I(R|P).\index{reverse relative entropy}$
\\
In the case of the relative and reverse relative entropies, the
global minimizer is $R$ and one can interpret  $I(P|R)$ and
$I(R|P)$ as some type of distances between $P$ and $R.$ Note also
that the positivity of the minimizers of \eqref{eq-24} is
guaranteed by $\{\gamma^*<\infty\}\subset [0,\infty).$
Consequently, the restriction that $P$ is a probability measure is
insured by the only unit mass constraint $P(\ZZ)=1.$

A very interesting Bayesian interpretation of \eqref{eq-24} is
obtained in \cite{DDCG90,GG97} in terms of the LDPs for $L_n$
defined at \eqref{eq-27}. The above entropies correspond to iid
weights $W$  distributed as follows:
\begin{enumerate}
    \item the relative entropy: Law($W$)=Poisson(1);
    \item the reverse relative entropy: Law($W$)=Exponential(1);
    \item the Fermi-Dirac entropy: Law($W$)=$(\delta_{-1}+\delta_{+1})/2;$
    \item  the $L_2$ norm: Law($W$)=Normal(0,1).
\end{enumerate}

\subsection*{Entropic and generalized entropic projections}
The minimizers of \eqref{eq-24} are called \emph{entropic
projections}. It may happen that even if the minimizer is not
attained, any minimizing sequence converges to some measure $Q_*$
which does not belong to $\CC.$ This intriguing phenomenon was
discovered by Csisz\'ar \cite{Csi84}. Such a $Q_*$ is called a
\emph{generalized} entropic projection.

In the special case where $I$ is the relative entropy, Csisz\'ar
has obtained existence results in \cite{Csi75} together with dual
equalities. His proofs are based on geometric properties of the
relative entropy; no convex analysis is needed. Based on the same
geometric ideas, he obtained later in \cite{Csi84} a powerful
Gibbs conditioning principle for noninteracting particles. For
general entropies as in \eqref{eq-25}, he studies the problem of
existence of entropic and generalized entropic projections in
\cite{Csi95}.

The minimization problem \eqref{eq-24} is interesting in its own
right, even when conditional LLNs are not at stake. The literature
on this subject is huge, for instance see \cite{Gzyl94,maxent} and
the references therein.

\subsection*{The extended minimization problem}
As will be seen, it is worth introducing an \emph{extension} of
\eqref{eq-24} to take the \emph{generalized} projections into
account. A solution to \eqref{eq-24} is in $A_\ZZ:$ the space of
all $Q\in\MZ$ which are absolutely continuous with respect to $R.$
We consider an extension $\Ib$ of the entropy $I$ to a vector
space $L_\ZZ$ which is the direct sum $L_\ZZ=A_\ZZ\oplus S_\ZZ$ of
$A_\ZZ$ and a vector space $S_\ZZ$ of \emph{singular linear forms}
(acting on numerical functions) which may not be
$\sigma$-additive. Any $\ell$ in $L_\ZZ$ is uniquely decomposed
into $\ell=\ell^a+\ell^s$ with $\ell^a\in A_\ZZ$ and $\ell^s\in
S_\ZZ$ and $\Ib$ has the following shape
\begin{equation*}
    \Ib(\ell)=I(\ell^a)+I^s(\ell^s)
\end{equation*}
where $I^s$ is a positively homogeneous function on $S_\ZZ.$ See
\eqref{III} for the precise description of $\Ib.$ For instance,
the extended relative entropy is
\begin{equation*}
\Ib(\ell| R) = I(\ell^a| R) + \sup\left\{\langle\ell^s,u\rangle ;
u, \IZ e^u\,dR<\infty\right\},\quad \ell\in L_\ZZ
\end{equation*}
and actually $\langle\ell^s,u\rangle=0$ for any $\ell$ such that $
\Ib(\ell)<\infty$ and any $u$ such that $\IZ e^{a|u|}\,dR<\infty$
for \emph{all} $a>0.$ The reverse entropy,
 $L_1$-norm and $L_1$-entropy also admit nontrivial extensions.
On the other hand, the extensions of the Fermi-Dirac, $L_p$-norm
and $L_p$-entropy with $p>1$ are trivial: $\{k\in
S_\ZZ;I^s(k)<\infty\}=\{0\}.$

The extended problem is
\begin{equation}\label{eq-24b}
     \textsl{minimize } \Ib(\ell) \textsl{ subject to }\ell\in \CC, \quad\ell\in L_\ZZ
\end{equation}
In fact, $\Ib$ is chosen to be the largest convex \lsc\ extension
of $I$ to $L_\ZZ$ with respect to some weak topology. This
guarantees tight relations between \eqref{eq-24} and
\eqref{eq-24b}. In particular, one can expect that their values
are equal for a large class of convex sets $\CC.$

Even if $I$ is strictly convex, $\Ib$ isn't strictly convex in
general since $I^s$ is positively homogeneous, so that
\eqref{eq-24b} may admit several minimizers.

Examples will be given of interesting situations where
\eqref{eq-24} is not attained in $A_\ZZ$ while \eqref{eq-24b} is
attained in $L_\ZZ.$

\subsection*{Dominating points}
Let the constraint set $\CC$ be described by
\begin{equation}\label{eq-29}
    \CC=\{Q\in\MZ; TQ\in C\}
    \index{C1@$\CC,$ constraint set}
    \index{C2@$C,$ constraint set}
\end{equation}
where $C$ is a subset of a vector space $\XX$ and $T:\MZ\to\XX$ is
a linear operator.  As a typical example, one can think of
\begin{equation}\label{eq-06}
TQ=\IZ\theta(z)\,Q(dz)
    \index{T2@$\theta,$ constraint function, see \eqref{eq-52}}
\end{equation}
where $\theta:\ZZ\to\XX$ is some function and the integral should
be taken formally for the moment. With $L_n$ given at
\eqref{eq-26} or \eqref{eq-27}, if $T$ is regular enough, we
obtain by the contraction principle that $X_n:=TL_n=\moyn
\theta(Z_i)\in\XX$ or $X_n:=TL_n=\moyn W_i\theta(z_i)$ obeys the
LDP in $\XX$ with rate function $J(x)=\inf\{I(Q);Q\in\MZ,TQ=x\},$
$x\in\XX.$ Once again, the conditional LLN for $X_n$ is of the
form: For ``all'' $A\subset\XX,$
$$
\Lim n\mathbb{P}(X_n\in A\mid X_n\in C)=\left\{%
\begin{array}{ll}
    1, & \hbox{if } A\ni x_* \\
    0, & \hbox{otherwise} \\
\end{array}%
\right.
$$
where $x_*$ is a solution to the minimization problem
\begin{equation}\label{eq-28}
      \textsl{minimize } J(x)\textsl{ subject to } x\in C,\quad x\in\XX
\end{equation}
The minimizers of \eqref{eq-28} are called \emph{dominating
points}. This notion was introduced by Ney \cite{Ney83, Ney84} in
the special case where  $\seq Zi$ is an iid sequence in $\ZZ=\R^d$
and $\theta$ is the identity, i.e.\! $X_n=\moyn Z_i.$  Later,
Einmahl and Kuelbs \cite{EK96, Kue00} have extended this study to
a Banach space $\ZZ.$ In this iid case, $J$ is the Cram\'er
transform of the law of $Z_1.$

\subsection*{Presentation of the results}
We treat the problems of existence of entropic projections and
dominating points in a unified way, taking advantage of the
mapping $TQ=x.$ Hence, we mainly concentrate on the entropic
projections and then transport the results to the dominating
points.

It will be proved at Proposition \ref{res:Iproj} that the entropic
projection exists on $\CC$ if the supporting hyperplanes of $\CC$
are directed by sufficiently integrable functions. In some cases
of not enough integrable supporting hyperplanes, the
representation of the generalized projection is still available
and given at Theorem \ref{res-05}.

It will appear that the generalized projection is the ``measure"
part of the minimizers of \eqref{eq-24b}. For instance, with the
relative entropy $I(.|R),$ the projection exists in $\CC$ if its
supporting hyperplanes are directed by functions $u$ such that
$\IZ e^{\alpha|u|}\,dR<\infty$ for \emph{all} $\alpha>0,$ see
Proposition \ref{res:entrop1}. If these $u$ only satisfy $\IZ
e^{\alpha|u|}\,dR<\infty$ for \emph{some} $\alpha>0,$ the
projection may not exist in $\CC,$ but the generalized projection
is computable: its Radon-Nykodym derivative with respect to $R$ is
characterized at Proposition \ref{res:entrop2}.

We find again some already known results of Csisz\'ar
\cite{Csi84,Csi95}, U.~Einmahl and Kuelbs \cite{EK96,Kue00} with
different proofs and a new point of view. The representations of
the generalized projections are new results. The conditions on $C$
to obtain dominating points are improved and an interesting
phenomenon noticed in \cite{Kue00} is clarified at Remark
\ref{rem-steep} by connecting it with the generalized entropic
projection.

Finally, a probabilistic interpretation of the singular components
of the generalized projections is proposed at Section
\ref{sec:condLLN}. It is obtained in terms of conditional LLNs.

The main results are Theorems \ref{res-07}, \ref{res-minseq},
\ref{res-05} and \ref{res:dompoint}.

\subsection*{Outline of the paper}
At Section \ref{sec-minent}, we give precise formulations of the
entropy minimization problems \eqref{eq-24} and \eqref{eq-24b}.
Then we recall at Theorems \ref{res-02} and \ref{res-03} results
from \cite{Leo08} about the existence and uniqueness of the
solutions of these problems, related dual equalities and the
characterizations of their solutions in terms of integral
representations.

Examples of standard entropies and constraints are presented at
Section \ref{sec-expl}.

We show at Theorem \ref{res-minseq} in Section \ref{sec:critical}
that under ``critical'' constraints, although the problem
\eqref{eq-24} may not be attained, its minimizing sequences may
converge in some sense to some measure $Q_*:$ the
\emph{generalized} entropic projection.

Section \ref{sec-entproj} is mainly a restatement of Sections
\ref{sec-minent} and \ref{sec:critical} in terms of entropic
projections. The results are also stated  explicitly for the
special important case of the relative entropy.

Section \ref{sec:dompoints} is devoted to dominating points. As
they are continuous images of entropic projections, the main
results of this section are corollaries of the results of Section
\ref{sec-entproj}.

At Section \ref{sec:condLLN}, we illustrate our results in terms
of conditional LLNs, see \eqref{eq-02}. In particular, the
generalized projections and singular components of the minimizers
of \eqref{eq-24b} are interpreted in terms of these conditional
LLNs.

\subsection*{Notation}
Let $X$ and $Y$ be topological vector spaces. The algebraic dual
space of $X$ is $X^{\ast},$ the topological dual space of $X$ is
$X'.$ The topology of $X$ weakened by $Y$ is $\sigma(X,Y)$ and we
write $\langle X,Y\rangle$ to specify that $X$ and $Y$ are in
separating duality.
\\
Let $f: X\rightarrow [-\infty,+\infty]$ be an extended numerical
function. Its convex conjugate\index{convexconj@convex conjugate,
$f^*,$ see \eqref{eq-61}} with respect to $\langle X,Y\rangle$ is
\begin{equation}\label{eq-61}
    f^*(y)=\sup_{x\in X}\{\langle x,y\rangle-f(x)\}\in [-\infty,+\infty],\quad y\in Y.
\end{equation}
Its subdifferential at $x$ with respect to $\langle X,Y\rangle$ is
$$\partial_Y f(x)=\{y\in Y; f(x+\xi)\geq f(x)+\langle
y,\xi\rangle, \forall
 \xi\in X\}.$$ If no confusion occurs, we write
$\partial f(x).$
\\
Let $A$ be a subset of $X,$ its intrinsic core is
    $
\icor A=\{x\in A; \forall x'\in\mathrm{aff}A, \exists t>0,
[x,x+t(x'-x)[\subset A\}
    \index{icor@$\icor$}
    $
where $\mathrm{aff}A$ is the affine space spanned by $A.$ Let us
denote $\dom f=\{x\in X; f(x)<\infty\}$ the effective domain of
$f$ and $\icordom f \index{icordom@$\icordom$}$ the intrinsic core
of $\dom f.$
\\
The convex indicator $\iota_A$ of a subset $A$ of $X$ is defined
at \eqref{eq-42}.
\\
We write
$$
I_\varphi(u):=\IZ \varphi(z,u(z))\,R(dz)=\IZ\varphi(u)\,dR
\index{I2@$I_\varphi,$ integral functional}
$$
and $I=I_{\gamma^*}$ for short, instead of \eqref{eq-25}.
\\
The inf-convolution\index{infconv@inf-convolution, $f\Box g$} of
$f$ and $g$ is $f\Box g(z)=\inf\{f(x)+g(y); x,y: x+y=z\}.$

An index of notation is provided at the end of the article.

\section{Minimizing entropy under convex constraints}\label{sec-minent}

In this section, the main results of \cite{Leo08} are recalled.
They are  Theorems \ref{res-02} and \ref{res-03} below and their
statements necessitate the notion of Orlicz spaces. First, we
recall basic definitions and notions about these function spaces
which are natural extensions of the standard $L_p$ spaces.

\subsection{Orlicz spaces} \label{sec:Orlicz}

The fact that the generalized projection may not belong to $\CC$
is connected with some properties of Orlicz spaces associated to
$I.$ Let us recall some basic definitions and results. A standard
reference is \cite{RaoRen}.

\newcommand{\Er}{E_\rho}
\newcommand{\Lr}{L_\rho}
\newcommand{\Lrs}{L_{\rho^*}}
\newcommand{\Lrsi}{L_{\rho}^s}
\newcommand{\la}{\ell^a}
\newcommand{\lsing}{\ell^s}

A set $\ZZ$ is furnished with a $\sigma$-finite nonnegative
measure $R$ on a $\sigma$-field which is assumed to be
$R$-complete. A function $\rho:\ZZ\times\R$ is said to be a
\emph{Young function}  \index{Orlicz spaces!rho@$\rho,$ Young
function}\index{Functions@Functions of $(s,z)$ or
$(t,z)$!rho@$\rho,$ Young function} if for $R$-almost every $z,$
$\rho(z,\cdot)$ is a convex even $[0,\infty]$-valued function on
$\R$ such that $\rho(z,0)=0$ and there exists a measurable
function $z\mapsto s_z>0$ such that $0<\rho(z,s_z)<\infty .$
\\
In the sequel, every numerical function on $\ZZ$ is supposed to be
measurable.

\begin{definitions}[The Orlicz spaces $\Lr$ and $\Er$] \index{Orlicz spaces}
The \emph{Orlicz space} associated with $ \rho $ is defined by
$\Lr=\{u:\ZZ\rightarrow\mathbb{R}; \Vert u\Vert_ \rho < + \infty
\}$ where the Luxemburg norm  \index{Orlicz
spaces!Luxemburg@$\Vert\cdot\Vert_ \rho,$ Luxemburg norm} $ \Vert
\cdot \Vert_ \rho $ is defined by
$$\Vert u\Vert_ \rho= \inf \left\{ \beta
>0\ ;\ \IZ\rho(z,u(z)/\beta)\, R(dz)\leq  1 \right\}$$ and $R$-a.e.\! equal functions are identified. Hence,
 $$
 \Lr=\left\{u:  \ZZ \rightarrow  \R \ ; \exists \alpha_o>0, \IZ
 \rho \Big(z,\alpha_o u(z)\Big)\,R(dz)<\infty \right\}.
  \index{Orlicz spaces!Lrho@$\Lr,$ Orlicz space}
 $$
 A subspace of interest is
 \begin{equation*}
     \Er=\left\{u:  \ZZ \rightarrow  \R \ ; \forall \alpha >0, \IZ
 \rho \Big(z,\alpha u(z)\Big)\,R(dz)<\infty \right\}.
   \index{Orlicz spaces!Erho@$\Er,$ small Orlicz space}
\end{equation*}
\end{definitions}

Taking $\rho_p(z,s)=|s|^p/p$ with $p\ge1$ gives
$L_{\rho_p}=E_{\rho_p}=L_p$ and the corresponding Luxemburg norm
is $ \Vert \cdot \Vert_{\rho_p} =p^{-1/p}\|\cdot\|_p$ where
$\|\cdot\|_p$ is the usual $L_p$ norm.

The convex conjugate $\rho^*$ of a Young function is still a Young
function, so that $L_{\rho^*}$ is also an Orlicz space. H\"older's
inequality in Orlicz spaces is
\begin{equation}\label{eq-43}
    \|uv\|_1\le 2\|u\|_{\rho}\|v\|_{\rho^*},\quad u\in L_\rho, v\in
    L_{\rho^*}.
\end{equation}
For instance, with $\rho_p^*(t)=|t|^q/q,$  $1/p+1/q=1,$
\eqref{eq-43} is $ \|uv\|_1\le
2p^{-1/p}q^{-1/q}\|u\|_{p}\|v\|_{q}.$ Note that
$2p^{-1/p}q^{-1/q}\ge1$ with equality when $p=2.$

\begin{theorem}[Representation of $\Er'$]\label{res-B1}
Suppose that $\rho$ is a finite Young function. Then, the dual
space of $\Er$ is isomorphic to $\Lrs,$ the Orlicz space
associated with the Young function $\rho^*$ which is the convex
conjugate of $\rho.$
\end{theorem}

 A continuous linear form $\ell\in\Lr'$ is said to be \emph{singular}
 if for all $u\in\Lr,$ there exists a decreasing sequence of
 measurable sets $(A_n)$ such that $R(\cap_n A_n)=0$ and for all
 $n\geq 1,$ $\langle \ell,u\1_{\ZZ\setminus A_n}\rangle=0.$
\begin{proposition}\label{res-B4}
 Let us assume that $\rho$ is finite. Then, $\ell\in\Lr'$ is
 singular if and only if $\langle \ell,u\rangle =0,$ for all $u$ in
 $\Er.$
 \end{proposition}
Let us denote respectively $\Lrs R=\{fR;f\in\Lrs\} \index{Orlicz
spaces!decomposition!Lrs@$\Lrs R,$ space of absolutely continuous
forms}$ and $\Lrsi \index{Orlicz spaces!decomposition!Lrs@$\Lrsi,$
space of singular forms}$ the subspaces of $\Lr'$ of all
absolutely continuous and singular forms.

\begin{theorem}
 [Representation of $\Lr'$]\label{res-B2}
 Let  $\rho $ be any Young function. The dual space of $\Lr$ is isomorphic to the direct sum
$
 \Lr'=\Lrs R\oplus \Lrsi.
 $
  This implies that any $\ell\in \Lr'$ is uniquely
 decomposed as
 \begin{equation}\label{decomp}
     \index{Orlicz spaces!decomposition!l1@$\la,$ absolutely continuous part of $\ell,$ see \eqref{decomp}}
     \index{Orlicz spaces!decomposition!l2@$\lsing,$ singular part of $\ell,$ see \eqref{decomp}}
 \ell=\la+\lsing
\end{equation}
with $\la\in\Lrs R$ and $\lsing\in\Lrsi.$
 \end{theorem}

In the decomposition \eqref{decomp}, $\la$ is called the
\emph{absolutely continuous} part of $\ell$ while $\lsing$ is its
\emph{singular part}.

The function $\rho $ is said to satisfy the $\Delta_2$-condition
\index{Orlicz spaces!Delta2@$\Delta_2$-condition, see
\eqref{delta2}} if
\begin{equation}
\label{delta2}
 \textrm{ there exist } \kappa>0,
 s_o\geq 0 \textrm{ such that } \forall s\geq s_o, z\in\ZZ, \rho_z (2s)\leq
 \kappa\rho_z (s)
\end{equation}
  When $R$ is bounded, in order that
$\Er=\Lr,$ it is enough that $\rho $ satisfies the
$\Delta_2$-condition.  Consequently, in this situation we have
$\Lr'=\Lrs
 R$ so that $\Lrsi$ reduces to the null vector space.
\\
We shall see below that the important case of the relative entropy
leads to a Young function $\rho$ which doesn't satisfy the
$\Delta_2$-condition; non-trivial singular components will appear.

\subsection{The assumptions}
Let $R$ be a positive measure\index{R@$R,$ reference measure} on a
space $\ZZ\index{ZZ@$\ZZ,$ reference space}$ and take a
$[0,\infty]$-valued measurable function
$\gamma^*$\index{Functions@Functions of $(s,z)$ or
$(t,z)$!gammastar@$\gamma^*,$ integrand, see \eqref{eq-25}} on
$\ZZ\times \mathbb{R}$ to define $I$ as at \eqref{eq-25}. In order
to define the constraint in a way similar to \eqref{eq-06}, take a
function
$$
    \theta:\ZZ\rightarrow\XXo\index{T2@$\theta,$ constraint function, see \eqref{eq-52}}
$$
where $\XXo\index{XXo@$\XXo,$ algebraic dual of $\YYo$}$ is the
algebraic dual space of some vector space
$\YYo.\index{YYa@$\YYo$}$ Let us collect the assumptions on
$R,\gamma^*,\theta$ and $C.$

\par\medskip\noindent\textbf{Assumptions (A).}\index{Assumptions (A)}\
\begin{itemize}
  \item[(A$_R$)]\index{Assumptions (A)!A1@(A$_R$) on the measure $R$}
    It is assumed that the reference measure $R$ is a bounded
    positive measure on a space $\ZZ$ endowed with some
    $R$-complete $\sigma$-field.
    \item[(A$_C$)]\index{Assumptions (A)!A2@(A$_C$) on the set $C$} $C$ is a convex subset of $\XXo.$
  \item[(A$_{\gamma^*}$)]\index{Assumptions (A)!A3@(A$_{\gamma^*}$) on $\gamma^*$} \emph{Assumptions on $\gamma^*.$}
     \begin{enumerate}
    \item
    $\gamma^*(\cdot,t)$ is $z$-measurable for all $t$ and
    for $R$-almost every $z\in \ZZ,$ $\gamma^*(z,\cdot)$
    is a \lsc\ strictly convex $[0,+\infty]$-valued function on
    $\mathbb{R}.$
    \item
    It is also assumed that for $R$-almost every $z\in \ZZ,$ $\gamma^*(z,\cdot)$
    attains a unique minimum denoted by $m(z),$\index{Functions@Functions of $(s,z)$ or $(t,z)$!m@$m,$ see (A$_{\gamma^*}$)}  the minimum value is
    $
    \gamma^*_z(m(z))=0,\ \forall z\in\ZZ
    $
    and there exist $a(z),b(z)>0$ such that
    $0<\gamma^*(z,m(z)+a(z))<\infty$ and
    $0<\gamma^*(z,m(z)-b(z))<\infty.$
    \item
    $\IZ \gamma^*((1+\alpha) m)\,dR+\IZ \gamma^*((1-\alpha)
    m)\,dR<\infty,$ for some $\alpha>0.$
     \end{enumerate}
 \end{itemize}
    The function  $\gamma^*_z$ is the convex conjugate of the
    \lsc\ convex function $\gamma_z=\gamma_z^{**}.\index{Functions@Functions
of $(s,z)$ or $(t,z)$!gamma@$\gamma,$ see \eqref{eq-49}}$ Defining
\begin{equation}\label{eq-49}
     \lambda(z,s)=\gamma(z,s)-m(z)s,\quad z\in \ZZ, s\in\mathbb{R},
     \index{Functions@Functions of $(s,z)$ or $(t,z)$!lambda@$\lambda,$ see \eqref{eq-49}}
\end{equation}
we see that for $R$-a.e.\! $z,$ $\lambda_z$ is a nonnegative
convex function and it vanishes at $s=0.$
 \begin{itemize}
  \item[(A$_\theta$)]\index{Assumptions (A)!A4@(A$_\theta$) on $\theta$} \emph{Assumptions on $\theta.$}
    \begin{enumerate}
    \item for any $y\in\YYo,$ the function $z\in\ZZ\mapsto\langle
    y,\theta(z)\rangle\in\R$ is measurable;
     \item for any $y\in\YYo,$ $\langle y,\theta(\cdot)\rangle=0, R\hbox{-a.e.}$ implies that $y=0;$
     \item[($\exists$)] $\forall y\in\YYo, \exists\alpha>0,\quad \IZ \lambda(\alpha\langle y, \theta\rangle)\,
        dR<\infty.$
    \end{enumerate}
\end{itemize}

Since
\begin{equation}\label{eq-50}
    \lmax(z,s)=\max[\lambda(z,s),\lambda(z,-s)]\in [0,\infty], \quad z\in \ZZ, s\in \mathbb{R}
    \index{Functions@Functions of $(s,z)$ or $(t,z)$!lambdamax@$\lmax,$ see \eqref{eq-50}}
\end{equation}
is a Young function, one can consider the corresponding Orlicz
spaces $\Ll\index{LL@$\Ll$}$ and $\Lls\index{Lls@$\Lls$}$ where
$\lmaxs(z,\cdot)$ is the convex conjugate of $\lmax(z,\cdot).$

\begin{remarks}[Some comments about these assumptions]\label{rem-01}\
\begin{enumerate}
    \item Assuming $R$ to be $\sigma$-finite is standard when working with
integral functionals. Here, $R$ is assumed to be bounded for
simplicity.
    \item Thanks to the theory of normal integrands \cite{Roc68},
$\gamma=\gamma^{**}$ is jointly measurable.
    \\ The function $m$ is measurable since its graph $\{(z,u): u=m(z)\}$ is
    measurable. Indeed, it is the $(z,u)$-projection of the intersection of
    the horizontal hyperplane $u=0$ with the measurable graph
    $\{(t,z,u):\gamma^*_z(t)=u\}$ of $\gamma^*.$ Note that we use
    the uniqueness assumption in (A$_{\gamma^*}^2$) to avoid a
    measurable selection argument.
    \\ As a consequence, $\lambda$ is a
    jointly measurable function.
    \item The assumption (A$_{\gamma^*}^3$) implies that  $m$ is in the Orlicz
    space $\Lls,$ see \eqref{eq-50}. In turn, the centered form $\ell
    -mR$ is in $\Ll'$ whenever $\ell$ is. This will be used later
    without warning. It also follows from \eqref{eq-49} that
    $\gamma^*_z(t)=\lambda^*_z(t-m(z)).$ These two facts allow us to write
    \begin{equation*}
    I_{\gamma^*}(fR)=I_{\lambda^*}(fR-mR),\quad f\in\Lls
\end{equation*}
where $\lambda^*_z$ is the convex conjugate of $\lambda_z.$
    \item The assumption \eqref{A-exists} is equivalent to $\langle
    y,\theta(\cdot)\rangle\in\Ll$ for all $y\in\YYo,$ see
    \eqref{A-critical} below.
\end{enumerate}
\end{remarks}

\subsection*{Examples}
We consider the entropies which were introduced after
\eqref{eq-42} and give the corresponding  functions $\lambda$ and
$\lmax.$
\begin{enumerate}
    \item The relative entropy: $\lambda(s)=e^s-s-1,$ $\lmax(s)=\lambda(|s|).$
    \item The reverse relative entropy:
    $\lambda(s)=-\log(1-s)-s+\iota_{\{s<1\}},$ $\lmax(s)=\lambda(|s|).$
    \item The Fermi-Dirac entropy: $\lambda(s)=\lmax(s)=\log\cosh s.$
    \item  The $L_p$ norm ($1<p<\infty$): $\lambda(s)=\lmax(s)=|s|^q/q$ with $1/p+1/q=1$ and
    \item the $L_p$ entropy ($1<p<\infty$): $\lambda(s)=\left\{\begin{array}{ll}
      0 & \textrm{if }s\le0 \\
      s^q/q & \textrm{if }s\ge0 \\
    \end{array}\right.,$ $\lmax(s)=|s|^q/q.$
\end{enumerate}

\subsection{The entropy minimization problems ($\pc$) and ($\pbc$)}

Because of the Remark \ref{rem-01}-(3), under (A$_{\gamma^*}^3$)
the effective domain of $I$ given at \eqref{eq-25} is included in
$\LlsR$ and the entropy functional to be considered is defined by
\begin{equation}\label{eq-16}
    I(fR)= \IZ \gamma^*(f)\,dR,\quad f\in\Lls.
        \index{I3@$I,$ entropy, see \eqref{eq-16}, \eqref{eq-08}}
\end{equation}
Assuming (A$_\theta$), that is
    \begin{equation}\label{A-critical}
    \langle\YYo,\theta(\cdot)\rangle\subset\Ll,
    \end{equation}
H\"older's inequality in Orlicz spaces, see \eqref{eq-43}, allows
to define the constraint operator
    $
    T:\ell\in\Ll'\mapsto\langle\theta,\ell\rangle\in\XXo
    $
by:
   \begin{equation}\label{eq-52}
\Big\langle y,\langle\theta,\ell\rangle\Big\rangle_{\YYo,\XXo}
 =\Big\langle\langle y,\theta\rangle ,\ell\Big\rangle_{\Ll,\Ll'},
  \forall y\in\YYo.
  \index{T1@$T,$ constraint operator, see \eqref{eq-52}}
  \index{T2@$\theta,$ constraint function, see \eqref{eq-52}}
\end{equation}
If $\ell=Q\in\LlsR\subset\MZ,$ one writes
$TQ=\langle\theta,Q\rangle=\IZ\theta\,dQ$ to mean \eqref{eq-52};
this is a $\XXo$-valued weak integral.
\begin{example}[Moment constraint] This is the easiest constraint
to think of. Let $ \theta=(\theta_k)_{1\leq k\leq K} $ be a
measurable function from $\ZZ$ to  $\XXo=\R^K.$ The moment
constraint is specified by the operator
$$
\IZ\theta\,d\ell=\left(\IZ\theta_k\,d\ell\right)_{1\leq k\leq
K}\in\R^K,
$$
which is defined for each $\ell\in\MZ$ which integrates all the
real valued measurable functions $\theta_k.$
\end{example}
The minimization problem \eqref{eq-24} becomes
\begin{equation}\label{pc}
      \textsl{minimize } I(Q)
    \textsl{ subject to } \IZ\theta\,dQ\in C,\quad Q\in \LlsR \tag{$\pc$}
    \index{Primal problems!P1@\eqref{pc}}
\end{equation}
where $C$ is a convex subset of $\XXo.$ The extended entropy is
defined by
\begin{equation}\label{III}
    \Ib(\ell)=I(\ell^a)+I^s(\ell^s),\quad \ell\in\Ll'
    \index{I4@$\Ib,$ extended entropy, see \eqref{III}, \eqref{eq-08}}
\end{equation}
where, using the notation of Theorem \ref{res-B2},
\begin{equation}\label{eq-51}
    I^s(\ell^s)=\iota^*_{\dom
I_\gamma}(\lsing)=\sup\left\{\langle\ell^s, u\rangle ; u\in\Ll,
I_\gamma(u)<\infty\right\}\in [0,\infty].
    \index{I5@$I^s,$ singular entropy, see \eqref{eq-51}}
\end{equation}
It is proved in \cite{Leo03} that $\Ib$ is the greatest convex
$\sigma(\Ll',\Ll)$-\lsc\ extension of $I$ to $\Ll'\supset\LlsR.$
The associated extended minimization problem is
\begin{equation}\label{pbc}
     \textsl{minimize } \Ib(\ell) \textsl{ subject to }
  \langle\theta,\ell\rangle\in C, \quad\ell\in\Ll'
  \tag{$\pbc$}
  \index{Primal problems!P2@\eqref{pbc}}
\end{equation}

\subsection{Good and critical constraints}\label{sec-const}

If the Young function $\lmax$ doesn't satisfy the
$\Delta_2$-condition  \eqref{delta2} as for instance with the
relative entropy and the reverse relative entropy (see items (1)
and (2) of the Examples above), the \emph{small} Orlicz space
$\El$ may be a proper subset of $\Ll.$ Consequently, for some
functions $\theta,$ the integrability property
\begin{equation}\label{A-good}
    \langle\YYo,\theta(\cdot)\rangle\subset\El
\end{equation}
or equivalently
\begin{equation}\label{A-forall}
 \forall y\in\YYo, \IZ \lambda(\langle y,
\theta\rangle)\,
  dR<\infty  \tag{A$_\theta^\forall$}
  \index{Assumptions (A)!A4@(A$_\theta$) on $\theta$!AA1@\eqref{A-forall}}
  \end{equation}
may not be satisfied while the weaker property
\begin{equation}\label{A-exists}
    \forall y\in\YYo, \exists\alpha>0, \IZ \lambda(\alpha\langle y,
\theta\rangle)\,
  dR<\infty \tag{A$^\exists_\theta$}
  \index{Assumptions (A)!A4@(A$_\theta$) on $\theta$!AA2@\eqref{A-exists}}
\end{equation}
 which is equivalent to \eqref{A-critical}, holds. In this situation, analytical
complications occur, see Section \ref{sec:critical}. This is the
reason why constraint satisfying \eqref{A-forall} are called
\emph{good constraints}, while constraints satisfying
\eqref{A-exists} but not \eqref{A-forall} are called
\emph{critical constraints}.

\subsection{Definitions of $\YY, \XX,$  $T^*,$ $\Gamma^*$ and ($\dc$)}\label{sec-def}
These objects will be necessary to state the relevant dual
problems. The general hypotheses (A) are assumed.

\par\smallskip\noindent \textsf{The space $\YY$.}\index{YY@$\YY,$ completion of $\YYo$}\
Because of the hypotheses (A$_\theta^2$) and \eqref{A-exists},
$\YYo$ can be identified with the subspace
$\langle\YYo,\theta(\cdot)\rangle$ of $\Ll.$ The space $\YY$ is
the extension of $\YYo$ which is isomorphic to the
$\|\cdot\|_\lmax$-closure of $\langle\YYo,\theta(\cdot)\rangle$ in
$\Ll.$

\par\smallskip\noindent \textsf{The space $\XX$.}\index{XX@$\XX,$ topological dual of $\YYo$ and $\YY$}\
The topological dual space of $\YY$ is $\XX=\YY'\subset\XXo.$
$\XX$ is identified with $\Ll'/\mathrm{ker\,}T.$
\\
Since $\lmax$ is finite, we can apply Proposition \ref{res-B4}. It
tells us that under the assumption \eqref{A-forall},
$T\Ll^s=\{0\}$ so that $\XX\cong\LlsR/\mathrm{ker\,}T.$

\par\smallskip\noindent\textsf{The operator $T^\ast.$}\index{T3@$T^\ast,$ adjoint of $T$}\
Let us define the adjoint $T^\sharp
:\XX^\ast\rightarrow\Ll^{\prime\ast}$ for all $\omega\in\XX^\ast$
by: $
  \langle  \ell,T^\sharp\omega\rangle_{\Ll',\Ll^{\prime\ast}}=\langle T\ell,\omega\rangle_{\XX,\XX^\ast},
  \forall \ell\in\Ll'.
$ We have the inclusions $\YYo\subset\YY\subset\XX^\ast.$ The
adjoint operator $T^*$ is the restriction of $T^\sharp$ to $\YYo.$
With some abuse of notation, we still denote $T^*y=\yt$ for
$y\in\YY.$ This can be interpreted as a dual bracket between
$\XXo^*$ and $\XXo$ since $T^*y=\langle\tilde{y},\theta\rangle$
$R$-a.e.\! for some $\tilde{y}\in\XXo^*.$

\par\smallskip\noindent\textsf{The function $\Gamma^\ast.$}\
 The basic dual problem associated with \eqref{pc} and
\eqref{pbc} is
\begin{equation*}
    \textsl{maximize } \inf_{x\in C}\yx-\Gamma(y),
   \quad y\in\YYo
\end{equation*}
where
\begin{equation}\label{eq-45}
    \Gamma(y) =I_\gamma(\yt),\quad y\in\YYo.
    \index{Gamma@$\Gamma,$ see \eqref{eq-45}}
\end{equation}
Let us denote
\begin{equation}\label{eq-53}
  \Gamma^*(x) =
  \sup_{y\in\YYo}\left\{\yx-I_\gamma(\yt)\right\},\quad x\in\XXo
  \index{Gammaast@$\Gamma^\ast,$ see \eqref{eq-53}}
\end{equation}
its convex conjugate. It is shown in \cite[Section 4]{p-Leo07a}
that $\dom\Gamma^*\subset\XX.$

\par\smallskip\noindent\textsf{The dual problem ($\dc$).}\
Another dual problem associated with \eqref{pc} and \eqref{pbc} is
\begin{equation}\label{dc}
    \textsl{maximize } \inf_{x\in \CX}\yx-I_\gamma(\yt),
   \quad y\in\YY\tag{$\dc$}
   \index{Dual problems!D1@\eqref{dc}}
\end{equation}

\subsection{Solving ($\pc$)}\label{sec-pc} In this section,
we study \eqref{pc} under the good constraint hypothesis
\eqref{A-forall} which imposes that $T^*\YY\subset \El$ and
$T(\LlsR)\subset\XX.$

\par\smallskip\noindent\textsf{The extended dual problem ($\dtc$).}\
The extended dual problem associated with \eqref{pc} is
\begin{equation}\label{dtc}
    \textsl{maximize } \inf_{x\in \CX}\langle\omega,x\rangle-I_\gamma(\langle\omega,\theta\rangle),
   \quad \omega\in\YYt\tag{$\dtc$}
   \index{Dual problems!D2@\eqref{dtc}}
\end{equation}
where $\YYt\subset\XX^*\index{YYt@$\YYt$}$ is some cone which
contains $\YYo.$ Its exact definition is given at  Appendix
\ref{annex}.

\begin{theorem}\label{res-02}
Suppose that
\begin{enumerate}
    \item the hypotheses \emph{(A)} and \eqref{A-forall} are satisfied;
    \item The convex set $C$ is assumed to be such that
 \begin{equation}\label{eq-05}
   T^{-1}C\cap\LlsR=\bigcap_{y\in Y}\left\{fR\in\LlsR;\IZ \yt f\,dR\ge a_y\right\}
\end{equation}
for some subset $Y\in\XXo^*$ with $\yt\in \El$ for all $y\in Y$
and some function $y\in Y\mapsto a_y\in\R.$ In other words,
$T^{-1}C\cap\LlsR$ is a $\sigma(\LlsR,\El)$-closed convex subset
of $\LlsR.$
\end{enumerate}
Then:
\begin{enumerate}[(a)]
 \item The dual equality for \eqref{pc} is
    \begin{equation*}
    \inf(\pc)=\sup(\dc)=\sup(\dtc)=\inf_{x\in C}\Gamma^*(x)\in [0,\infty].
    \end{equation*}
 \item If $C\cap \dom\Gamma^*\not=\emptyset$ or equivalently $C\cap T\dom I\not=\emptyset,$
 then \eqref{pc}
 admits a unique solution $\Qh\index{Qh@$\Qh,$ minimizer of \eqref{pc}}$ in $\LlsR$ and any minimizing sequence $\seq Qn$ converges to
 $\Qh$ with respect to the topology   $\sigma(\LlsR,\Ll).$
\end{enumerate}
 Suppose that in addition $C\cap \icordom\Gamma^*\not=\emptyset$ or equivalently $C\cap \icor(T\dom I)\not=\emptyset.$
\begin{enumerate}[(a)]
 \item[(c)] Let us define $\xh\index{xh@$\xh,$ see \eqref{eq-36}, \eqref{eq-40}}:= \IZ\theta\,d\Qh$ in the weak sense with respect to
    the duality $\langle\YYo,\XXo\rangle.$
 There exists  $\ot\in\YYt\index{o1@$\ot,$ see \eqref{eq-36}}$ such that
 \begin{equation}\label{eq-36}
    \left\{\begin{array}{cl}
      (a) & \xh\in C\cap\dom\Gamma^* \\
      (b) & \langle \ot,\xh\rangle_{\XX^*,\XX} \leq \langle \ot,x\rangle_{\XX^*,\XX}, \forall x\in C\cap\dom\Gamma^* \\
      (c) & \Qh(dz)=\gamma'_z(\langle\ot, \theta(z)\rangle)\,R(dz). \\
    \end{array}\right.
\end{equation}
Furthermore, $\Qh\in\LlsR$ and $\ot\in\YYt$ satisfy \eqref{eq-36}
if and only if $\Qh$ solves \eqref{pc} and $\ot$ solves
\eqref{dtc}.
    \item[(d)]
    Of course, \eqref{eq-36}-c implies
\begin{equation}\label{eq:rep-xb}
    \xh=\IZ \theta\gamma'(\ott)\,dR
\end{equation}
in the weak sense.
    Moreover,
    \begin{enumerate}[1.]
        \item $\xh$ minimizes $\Gamma^*$ on $C,$
        \item
        $I(\Qh)=\Gamma^*(\xh)=\IZ\gamma^*\circ\gamma'(\ott)\,dR<\infty$
        and
        \item $I(\Qh)+\IZ\gamma(\ott)\,dR=\IZ\ott\,d\Qh.$
    \end{enumerate}
\end{enumerate}
\end{theorem}

\begin{proof}
This result is \cite[Theorem 3.2]{Leo08}.
\end{proof}

Figure \ref{fig1} illustrates the items (a) and (b) of
\eqref{eq-36} which, with $T\Qh=\xh,TQ=x\in C,$ can be rewritten:
$\Qh\in T^{-1}C$ and $\langle T^*\ot, Q-\Qh\rangle\ge0, \forall
Q\in T^{-1}C:$ the shaded area.

\begin{figure}[h]
\begin{center}
\scalebox{1} 
{
\begin{pspicture}(0,-1.72)(12.442813,1.72)
\definecolor{color5b}{rgb}{0.8,0.8,0.8}
\psellipse[linewidth=0.04,dimen=outer](5.4209375,-0.58)(1.6,0.62)
\psellipse[linewidth=0.04,dimen=outer](6.6409373,-0.72)(3.42,1.0)
\psellipse[linewidth=0.04,dimen=outer,fillstyle=solid,fillcolor=color5b](8.940937,0.88)(1.2,0.84)
\psline[linewidth=0.03cm](4.8209376,1.0)(11.700937,-0.64)
\psdots[dotsize=0.12](4.6209373,-0.6)
\psdots[dotsize=0.15](8.580937,0.08)
\psline[linewidth=0.03cm,arrowsize=0.05291667cm
4.0,arrowlength=1.4,arrowinset=0.4]{->}(10.900937,-0.44)(11.300938,1.08)
\psdots[dotsize=0.12](8.180938,0.92)
\psline[linewidth=0.03cm,arrowsize=0.05291667cm
4.0,arrowlength=1.4,arrowinset=0.4]{->}(8.540937,0.12)(8.180938,0.92)
\usefont{T1}{ptm}{m}{n}
\rput(5.0223436,-0.77){$R$}
\usefont{T1}{ptm}{m}{n}
\rput(8.612344,-0.33){$\Qh$}
\usefont{T1}{ptm}{m}{n}
\rput(8.672344,0.91){$Q$}
\usefont{T1}{ptm}{m}{n}
\rput(7.1923437,1.51){$T^{-1}C$}
\usefont{T1}{ptm}{m}{n}
\rput(11.742344,0.07){$T^*\ot$}
\psbezier[linewidth=0.024,arrowsize=0.05291667cm
2.0,arrowlength=1.4,arrowinset=0.4]{->}(2.4209375,-0.32)(3.2209375,0.08)(3.2609375,-0.08)(3.9409375,-0.4)
\psbezier[linewidth=0.024,arrowsize=0.05291667cm
2.0,arrowlength=1.4,arrowinset=0.4]{->}(2.4609375,-0.32)(2.9809375,-0.4)(3.1409376,-0.32)(3.4209375,-0.44)
\usefont{T1}{ptm}{m}{n}
\rput(1.4223437,-0.25){$I=\textrm{const.}$}
\end{pspicture}
}
\end{center}
 \caption{}\label{fig1}
 \end{figure}
As $R$ is the unconstrained minimizer of $I,$ the sub-level sets
of $I$ form an increasing family of convex sets  which expends
from $R.$ One sees that the hyperplane $\{Q: \langle T^*\ot,
Q-\Qh\rangle=0\}$ separates the convex set $T^{-1}C$ and the
convex level set of $I$ which is ``tangent" to $T^{-1}C.$
Informally, the gradient of $I$ at $\Qh$ is a normal vector of
$T^{-1}C$ at $\Qh:$ it must be a multiple of $T^*\ot.$ Expressing
this orthogonality relation by means of the convex conjugation
 \eqref{eq-53} leads to the representation formula
\eqref{eq-36}-(c). Clearly, Hahn-Banach theorem has something to
do with the existence of $T^*\ot.$ This picture is a guide for the
proof in \cite{Leo08}.

 Following the terminology of Ney \cite{Ney83,Ney84}, as
it shares the properties \eqref{eq-36}-(a,b) and
\eqref{eq:rep-xb}, the minimizer $\xh$ is called a
\emph{dominating point} of $C$ with respect to $\Gamma^*$ (see
Definition \ref{defdomp} below).

\subsection{Solving ($\pbc$)}\label{sec-pbc}
In this section, we study \eqref{pbc} under the critical
constraint hypothesis \eqref{A-exists} which imposes that
$T^*\YY\subset \Ll.$

By Theorem \ref{res-B2}, we have $L_\rho''=[L_{\rho}\oplus
L_{\rho^*}^s]\oplus L_{\rho}^{s\prime}.$ For any $\zeta\in
L_{\rho}''=( L_{\rho^*}R\oplus L_{\rho}^s)',$ let us denote the
restrictions $\zeta_1=\zeta_{| L_{\rho^*}R}$ and $\zeta_2=\zeta_{|
L_{\rho}^s}.$ Since, $( L_{\rho^*}R)'\simeq L_{\rho}\oplus
L_{\rho^*}^s,$ we see that any $\zeta\in L_{\rho}''$ is uniquely
decomposed into
\begin{equation}\label{eq-dec}
    \zeta=\zeta_1^a+\zeta_1^s+\zeta_2
    \index{Orlicz spaces!decomposition!l3@$\zeta_1^a,\zeta_1^s,\zeta_2,$ see \eqref{eq-dec}}
\end{equation}
with $\zeta_1=\zeta_1^a+\zeta_1^s\in L_{\rho^*}',$ $\zeta_1^a\in
L_{\rho},$ $\zeta_1^s\in L_{\rho^*}^s$ and $\zeta_2\in
L_{\rho}^{s\prime}.$

\par\smallskip\noindent\textsf{The extended dual problem ($\dbc$).}\
 The extended dual problem associated with
\eqref{pbc} is
\begin{equation}\label{dbc}
   \textsl{maximize } \inf_{x\in \CX}\langle\omega,x\rangle-
   I_\gamma\big([T^*\omega]_1^a\big),
   \quad \omega\in\YYb \tag{$\dbc$}
   \index{Dual problems!D3@\eqref{dbc}}
\end{equation}
where $\YYb\subset\XX^*\index{YYb@$\YYb$}$ is some cone which
contains $\YYo.$ Its exact definition is given at  Appendix
\ref{annex}.

\begin{theorem}\label{res-03}
Suppose that
\begin{enumerate}
    \item the hypotheses \emph{(A)} are satisfied;
    \item The convex set $C$ is assumed to be such that
 \begin{equation}\label{eq-10}
   T^{-1}C\cap\Ll'=\bigcap_{y\in Y}\left\{\ell\in\Ll';\langle\yt,\ell\rangle \ge a_y\right\}
\end{equation}
for some subset $Y\subset\XXo^*$ with $\yt\in\Ll$ for all $y\in Y$
and some function $y\in Y\mapsto a_y\in\R.$ In other words,
$T^{-1}C$ is a $\sigma(\Ll',\Ll)$-closed convex subset of $\Ll'.$
\end{enumerate}
Then:
\begin{enumerate}[]
\item[(a)] The dual equality for \eqref{pbc} is
\begin{equation*}
    \inf(\pbc)=\inf_{x\in C}\Gamma^*(x)=\sup(\dc)=\sup(\dbc)\in [0,\infty].
    \end{equation*}
 \item[(b)] If $C\cap \dom\Gamma^*\not=\emptyset$ or equivalently $C\cap T\dom \Ib\not=\emptyset,$
 then \eqref{pbc}
 admits  solutions in $\Ll',$ any minimizing sequence admits $\sigma(\Ll',\Ll)$-cluster points
 and every such point is a solution to \eqref{pbc}.
\end{enumerate}
 Suppose that in addition we have
    $
    C\cap \icordom\Gamma^*\not=\emptyset
    $
 or equivalently $C\cap \icor(T\dom \Ib)\not=\emptyset.$ Then:
\begin{enumerate}[]
 \item[(c)]
 Let $\lh\in\Ll'\index{lh@$\lh,$ solution to \eqref{pbc}, see \eqref{eq-40}}$ be any solution to \eqref{pbc}
 and denote $\xh:= T\lh\index{xh@$\xh,$ see \eqref{eq-36}, \eqref{eq-40}}.$ There exists
    $\ob\in\YYb\index{o2@$\ob,$ see \eqref{eq-40}, \eqref{eq-40bis}}$
such that
 \begin{equation}\label{eq-40}
    \left\{\begin{array}{cl}
      (a) & \xh\in C\cap\dom\Gamma^* \\
      (b) & \langle \ob,\xh\rangle_{\XX^*,\XX} \leq \langle \ob,x\rangle_{\XX^*,\XX}, \forall x\in C\cap\dom\Gamma^* \\
      (c) & \lh\in\gamma'_z([T^*\ob]_1^a)\,R+D^\bot([T^*\ob]_2) \\
    \end{array}\right.
\end{equation}
where we used notation \eqref{eq-dec} and $D^\bot([T^*\ob]_2)$ is
some cone in $\Ll^s$ which is pointed at zero and whose direction
depends on $[T^*\ob]_2$ (see Appendix \ref{annex} for its precise
definition).
\\
There exists some $\ot\in\XXo^*$ such that
$$
[T^*\ob]_1^a=\langle\ot,\theta(\cdot)\rangle_{\XXo^*,\XXo}
$$
is a measurable function which can be approximated in some sense
(see Appendix \ref{annex}) by sequences $(\langle
y_n,\theta(\cdot)\rangle)_{n\ge1}$ with $y_n\in\YYo$ and
$\IZ\lambda(\langle y_n,\theta\rangle)\,dR<\infty$ for all $n.$
\\
Furthermore, $\lh\in\Ll'$ and $\ob\in\YYb$ satisfy \eqref{eq-40}
if and only if $\lh$ solves \eqref{pbc} and $\ob$ solves
\eqref{dbc}.
    \item[(d)]
     Of course, \eqref{eq-40}-c implies
    $
\xh=\IZ \theta\gamma'(\ott)\,dR+\langle\theta,\lh^s\rangle.
    $
    Moreover,
    \begin{enumerate}[1.]
        \item $\xh$ minimizes $\Gamma^*$ on $C,$
        \item
        $\Ib(\lh)=\Gamma^*(\xh)=\IZ\gamma^*\circ\gamma'(\ott)\,dR+\sup\{\langle u,\lh^s\rangle;u\in\dom I_\gamma\}<\infty$
        and
        \item $\Ib(\lh)+\IZ\gamma(\ott)\,dR=\IZ\ott\,d\lh^a+\langle [T^*\ob]_2,\lh^s \rangle.$
    \end{enumerate}
\end{enumerate}
\end{theorem}

\begin{proof}
This result is \cite[Theorem 4.2]{Leo08}.
\end{proof}

The exact definitions of $\YYt,$  $\YYb$ and $D^\bot$ as well as
the precise statement of Theorem \ref{res-03}-(c) are given at
Appendix \ref{annex}. In particular, the complete statement of
Theorem \ref{res-03}-(c) is given at Theorem \ref{res-03bis}.

Figure \ref{fig2} illustrates \eqref{eq-40} which, with $\xh=T\lh$
can be rewritten: $\lh\in T^{-1}C$ and $\langle T^*\ob,
\ell-\lh\rangle\ge0, \forall \ell\in T^{-1}C.$ As in Figure
\ref{fig1}, one sees that the hyperplane $\{\ell: \langle T^*\ob,
\ell-\lh\rangle=0\}$ separates the convex set $T^{-1}C$ and the
convex level set of $\Ib$ which is ``tangent" to $T^{-1}C:$ the
shaded area. The same kind of conclusions follow.

\begin{figure}[ht]
\begin{center}
\scalebox{1} 
{
\begin{pspicture}(0,-3.5831075)(12.842813,3.6081078)
\definecolor{color331b}{rgb}{0.8,0.8,0.8}
\psbezier[linewidth=0.024,fillstyle=solid,fillcolor=color331b](5.7609377,-0.010329752)(5.747474,0.012433295)(6.3869605,1.8430234)(7.3609376,2.0696702)(8.334914,2.296317)(9.728796,1.6295751)(10.120937,0.70967025)(10.513079,-0.21023464)(8.071575,-0.51840436)(8.020938,-0.5703297)(7.9703,-0.6222551)(5.7744007,-0.0330928)(5.7609377,-0.010329752)
\psline[linewidth=0.03cm](4.9409375,0.14967024)(11.820937,-1.4903297)
\psdots[dotsize=0.12](5.9409375,-1.9303298)
\psdots[dotsize=0.15](7.6209373,-0.49032974)
\psline[linewidth=0.03cm,arrowsize=0.05291667cm
4.0,arrowlength=1.4,arrowinset=0.4]{->}(11.020938,-1.2903297)(11.280937,-0.37032974)
\usefont{T1}{ptm}{m}{n}
\rput(5.5823436,-2.0203297){$R$}
\usefont{T1}{ptm}{m}{n}
\rput(7.722344,-0.18032975){$\lh$}
\usefont{T1}{ptm}{m}{n}
\rput(5.432344,0.81967026){$T^{-1}C$}
\usefont{T1}{ptm}{m}{n}
\rput(12.102344,-0.98032975){$T^*\ob$}
\psbezier[linewidth=0.024,arrowsize=0.05291667cm
2.0,arrowlength=1.4,arrowinset=0.4]{->}(2.5809374,-1.1703298)(3.3809376,-0.7703298)(3.7409375,-1.2103298)(4.4609375,-1.5303297)
\psbezier[linewidth=0.024,arrowsize=0.05291667cm
2.0,arrowlength=1.4,arrowinset=0.4]{->}(2.5809374,-1.1703298)(3.1009376,-1.2503297)(3.2609375,-1.1703298)(3.5409374,-1.2903297)
\usefont{T1}{ptm}{m}{n}
\rput(1.5623437,-1.1003298){$\Ib=\textrm{const.}$}
\psbezier[linewidth=0.03](6.765821,-0.29032975)(6.644788,-0.27915227)(3.651193,-0.2587999)(3.5211258,-1.2503297)(3.3910584,-2.2418597)(4.8475676,-3.4433591)(7.126343,-3.4503298)(9.405118,-3.4573004)(8.840938,-0.7703298)(8.420938,-0.69032973)(8.000937,-0.61032975)(6.886854,-0.30150723)(6.765821,-0.29032975)
\psline[linewidth=0.03cm,arrowsize=0.05291667cm
4.0,arrowlength=1.4,arrowinset=0.4]{->}(7.3609376,-2.1703298)(10.860937,3.1496704)
\psline[linewidth=0.03cm,arrowsize=0.05291667cm
4.0,arrowlength=1.4,arrowinset=0.4]{->}(7.3609376,-2.1703298)(5.9009376,3.3496702)
\psarc[linewidth=0.03](8.080937,0.069670245){2.88}{44.409344}{131.03534}
\usefont{T1}{ptm}{m}{n}
\rput(8.052343,3.4196703){$D^\bot$}
\psdots[dotsize=0.12](7.3609376,-2.1703298)
\usefont{T1}{ptm}{m}{n}
\rput(7.132344,-2.4603298){$\Qd$}
\psline[linewidth=0.078cm,tbarsize=0.07055555cm
5.0]{|*-|*}(6.8409376,-0.29032975)(7.9409375,-0.5703297)
\psbezier[linewidth=0.03](6.9809375,-0.97032976)(6.9409375,-0.97032976)(4.3809376,-1.1712627)(4.3809376,-1.6503297)(4.3809376,-2.129397)(7.0711684,-3.5681078)(7.9409375,-3.0103297)(8.810707,-2.4525518)(8.020938,-1.2103298)(7.9809375,-1.2103298)(7.9409375,-1.2103298)(7.0209374,-0.97032976)(6.9809375,-0.97032976)
\end{pspicture}
}
\end{center}
 \caption{}\label{fig2}
 \end{figure}

The measure $\Qd=\lh^a=\gamma'_z([T^*\ob]_1^a)\,R$ is the
absolutely continuous part of $\lh;$ its expression is given at
\eqref{eq-18} below. The set of solutions $\lh$ of (\ref{pbc}) is
represented by the bold type segment, they all share the same
absolutely continuous part $\Qd.$  The relation \eqref{eq-40}-(c)
can be rewritten: $\lh\in\Qd +D^\bot$ which is the convex cone
with vertex $\Qd$ and direction $D^\bot.$ The cone $D^\bot$ only
contains singular directions. One sees with \eqref{III} and
\eqref{eq-51} that $D^\bot$ contributes to the positively
homogeneous part $I^s$ of $\Ib.$ This is the reason why we decided
to draw a flat part for the level lines of $\Ib$ when they cut
$\Qd+D^\bot.$
\\
Figure \ref{fig2} is only a guide for illustrating Theorem
\ref{res-03}. It shouldn't be taken too seriously. For instance a
better finite dimensional analogue would have been given by
$\Ib(x,y,z)=x^2+|y+z|$ which requires a 3D graphical
representation.

\section{Some examples}\label{sec-expl}

\subsection{Some examples of entropies}
Important examples of entropies occur in statistical physics,
probability theory and mathematical statistics. Among them the
relative entropy plays a prominent role.

\subsubsection*{Relative entropy}  The reference measure $R$ is assumed to be
a probability measure. The relative entropy of $Q\in\MZ$ with
respect to $R\in\PZ$ is
\begin{equation}\label{eq-54}
    I(Q|R)=\left\{%
\begin{array}{ll}
    \IZ \log\left(\frac{dQ}{dR}\right)\,dQ & \hbox{if }Q\prec R \hbox{ and }Q\in\PZ\\
    +\infty & \hbox{otherwise} \\
\end{array}%
\right.,\quad Q\in\MZ.
    \index{I6@$I(\cdot\mid R),$ relative entropy, see \eqref{eq-54}}
\end{equation}
It corresponds to $\gamma^*_z(t)=\left\{%
\begin{array}{ll}
    t\log t-t+1 & \hbox{if }t>0 \\
     1 & \hbox{if }t=0 \\
    +\infty & \hbox{if }t<0 \\
\end{array}%
\right.,   $ $m(z)=1$ and
\begin{equation}\label{eq-01}
     \lambda_z(s)=e^s-s-1,\quad s\in\R, z\in \ZZ.
\end{equation}

\subsubsection*{A variant} Taking the same $\gamma^*$ and removing the unit mass constraint gives
\begin{equation*}
    H(Q|R)=\left\{%
\begin{array}{ll}
    \IZ \left[\frac{dQ}{dR}\log\left(\frac{dQ}{dR}\right)-\frac{dQ}{dR}+1\right]\,dR & \hbox{if }0\le Q\prec R \\
    +\infty, & \hbox{otherwise} \\
\end{array}%
\right.,\quad Q\in\MZ
\end{equation*}
This entropy is the rate function of \eqref{eq-27} when $\seq Wi$
is an iid sequence of Poisson(1) random weights. If $R$ is
$\sigma$-finite, it is the rate function of the LDP of normalized
Poisson random measures, see \cite{Leo00}.

\subsubsection*{Extended relative entropy}

Since  $\lambda(s)=e^s-s-1$ and $R\in\PZ$ is a bounded measure, we
have $\lmax(s)=\tau(s):=e^{|s|}-|s|-1
    \index{Functions@Functions of $(s,z)$ or $(t,z)$!tau@$\tau$}$
and the relevant Orlicz spaces are
\begin{eqnarray*}
  L_{\tau^*} &=& \{f:\ZZ\rightarrow\mathbb{R}; \IZ |f|\log |f|\,dR<\infty\} \\
  E_\tau &=&  \{u:\ZZ\rightarrow\mathbb{R};\forall\alpha>0, \IZ e^{\alpha |u|}\,dR<\infty\}\\
 L_\tau &=&  \{u:\ZZ\rightarrow\mathbb{R};\exists\alpha>0, \IZ e^{\alpha |u|}\,dR<\infty\}
 \index{Ltaustar@$L_{\tau^*}$}\index{Ltau@$L_\tau$}\index{Etau@$E_\tau$}
\end{eqnarray*}
since $\tau^*(t)=(|t|+1)\log(|t|+1)-|t|.\index{Functions@Functions
of $(s,z)$ or $(t,z)$!taustar@$\tau^*$}$ The extended relative
entropy is defined by
\begin{equation}\label{ext-entrop}
\Ib(\ell| R) = I(\ell^a| R) + \sup\left\{\langle\ell^s,u\rangle ;
u, \IZ e^u\,dR<\infty\right\},\quad \ell\in\OZ
    \index{I7@$\Ib(\cdot\mid R),$ extended relative entropy, see \eqref{ext-entrop}}
\end{equation}
where $\ell=\ell^a+\ell^s$ is the decomposition into absolutely
continuous and singular parts of $\ell$ in
$L_\tau'=L_{\tau^*}\oplus L_\tau^s,$ and
\begin{equation}\label{eq-44}
    \OZ=\{\ell\in L_\tau' ; \ell\geq 0, \langle\ell,\1\rangle=1\}.
    \index{Oexp@$\OZ,$ see \eqref{eq-44}}
\end{equation}
Note that $\OZ$ depends on $R$ and that for all $\ell\in\OZ,$
$\ell^a\in\PZ\cap L_{\tau^*}R.$

\subsection{Some examples of constraints}\label{sec:constraints}

Let us consider the two standard constraints which are the moment
constraints and the marginal constraints.

\subsubsection*{Moment constraints}
Let $ \theta=(\theta_k)_{1\leq k\leq K} $ be a measurable function
from $\ZZ$ to  $\XXo=\R^K.$ The moment constraint is specified by
the operator
$$
\IZ\theta\,d\ell=\left(\IZ\theta_k\,d\ell\right)_{1\leq k\leq
K}\in\R^K,
$$
which is defined for each $\ell\in\MZ$ which integrates all the
real valued measurable functions $\theta_k.$ The adjoint operator
is
\begin{equation*}
   T^*y= \langle y,\theta\rangle=\sum_{1\leq k\leq K} y_k\theta_k,\quad
    y=(y_1,\dots,y_K)\in\R^K.
\end{equation*}

\subsubsection*{Marginal constraints}
Let $\ZZ=\AB$ be a product space, $M_{AB}$ be the space of all
\emph{bounded} signed measures on $\AB$ and $U_{AB}$ be the space
of all measurable bounded functions $u$ on $\AB.$ Denote
$\ell_A=\ell(\cdot\times B)$ and $\ell_B=\ell(A\times\cdot)$ the
marginal measures of $\ell\in M_{AB}.$ The constraint of
prescribed marginal measures is specified by
$$
\IAB\theta\,d\ell=(\ell_A,\ell_B)\in M_{A}\times M_{B},\quad
\ell\in M_{AB}
$$
where $M_{A}$ and $M_{B}$ are the spaces of all bounded signed
measures on $A$ and $B.$ The function $\theta$ which gives the
marginal constraint is
$$
\theta(a,b)=(\delta_{a}, \delta_{b}),\ a\in A, b\in B
$$
where $\delta_a$ is the Dirac measure at $a.$ Indeed,
$(\ell_A,\ell_B)=\IAB (\delta_a,\delta_b)\,\ell(dadb).$
\\
More precisely, let $U_{A},$ $U_{B}$ be the spaces of measurable
functions on $A$ and $B$ and take $\YYo=U_{A}\times U_{B}$ and
$\XXo=U_{A}^*\times U_{B}^*.$ Then, $\theta$ is a measurable
function from $\ZZ=\AB$ to $\XXo=U_{A}^*\times U_{B}^*$ and the
adjoint of the marginal operator
\begin{equation*}
   T\ell= \langle\theta,\ell\rangle=(\ell_A,\ell_B)\in U_{A}^*\times U_{B}^*,\quad
    \ell\in U_{AB}^*
\end{equation*}
where $\langle f,\ell_A\rangle:=\langle f\otimes 1,\ell\rangle$
and $\langle g,\ell_B\rangle:=\langle 1\otimes g,\ell\rangle$ for
all $f\in U_A$ and $g\in U_B,$ is given by
\begin{equation}\label{eq-19}
   T^*(f,g)= \langle(f,g),\theta\rangle=f\oplus g\in U_{AB},\quad f\in U_{A}, g\in U_{B}
\end{equation}
where $f\oplus g(a,b):=f(a)+g(b),$ $a\in A, b\in B.$

\section{Minimizing sequences}\label{sec:critical}

In this section, the minimization problem \eqref{pc} is considered
when the constraint function $\theta$ satisfies \eqref{A-exists}
but not necessarily \eqref{A-forall}. This means that the
constraint is \emph{critical.} Problem \eqref{pc} may not be
attained anymore. Nevertheless, minimizing sequences may admit a
limit in some sense. As will be seen at Section \ref{sec-entproj},
this phenomenon is tightly linked to the notion of
\emph{generalized entropic projection} introduced by Csisz\'ar.

\subsection{Statements of the results} We start this section stating its
main results at Theorems \ref{res-07} and \ref{res-minseq}.

\begin{theorem}[Attainment]\label{res-07}
The hypotheses of Theorem \ref{res-03} are assumed.
\begin{enumerate}[(a)]
    \item Suppose that $C\cap\dom\Gamma^*\not=\emptyset.$
 Then, the minimization problem \eqref{pbc} is attained in $\Ll'$
and all its solutions share the same absolutely continuous part
$\Qd\in\LlsR.$
    \item Suppose that $C\cap\icordom\Gamma^*\not=\emptyset.$ Then,
   \eqref{dtc} is attained in $\YYt$ and
\begin{equation}\label{eq-18}
     \Qd(dz)=\gamma'_z(\langle\od, \theta(z)\rangle)\,R(dz)
     \index{Qd@$\Qd,$ see \eqref{eq-18}}
\end{equation}
where $\od\in\YYt\index{od@$\od,$ see \eqref{eq-18}}$ is any
solution to \eqref{dtc}.
    \item See also an additional statement at Proposition
    \ref{res-07bis}.
\end{enumerate}
\end{theorem}

\begin{proof}
\boulette{(a)} The attainment  statement is Theorem
\ref{res-03}-b. Let us show that as $\gamma^*$ is strictly convex,
if $k_*$ and $\ell_*$ are two solutions of \eqref{pbc}, their
absolutely continuous parts match:
\begin{equation}\label{eq-17}
k_*^a=\ell_*^a.
\end{equation}
$k_*, \ell_*$ are in the convex set $\{\ell\in \Ll'; T\ell\in C\}$
and $\inf(\pbc)=\Ib(k_*)=\Ib(\ell_*).$ For all $0\leq p, q\leq 1$
such that $p+q=1,$ as $I$ and $I^s$ are convex functions, we have
\begin{eqnarray*}
  \inf(\pbc) &\leq& \Ib(pk_*+q\ell_*)\\
    &=& I(pk_*^a+q\ell_*^a)+I^s(pk_*^s+q\ell_*^s)\\
    &\leq& pI(k_*^a)+qI(\ell_*^a)+pI^s(k_*^s)+qI^s(\ell_*^s)\\
    &=& p\Ib (k_*)+q\Ib(\ell_*)=\inf(\pbc)
\end{eqnarray*}
It follows that
$I(pk_*^a+q\ell_*^a)+I^s(pk_*^s+q\ell_*^s)=pI(k_*^a)+qI(\ell_*^a)
+pI^s(k_*^s)+qI^s(\ell_*^s).$ Suppose that $k_*^a\not =\ell_*^a.$
As $I$ is strictly convex, with $0<p<1,$ we get:
$I(pk_*^a+q\ell_*^a)<pI(k_*^a)+qI(\ell_*^a)$ and this implies that
$I^s(pk_*^s+q\ell_*^s)>pI^s(k_*^s)+qI^s(\ell_*^s)$ which is
impossible since $I^s$ is convex. This proves \eqref{eq-17}.

    \Boulette{(b)}
Let $\lb$ be any solution to ($\pbc$). Denoting
$\xb^a=T\lb^a\index{xxx@$\xb^a, \xb^s$}$ and $\xb^s=T\lb^s$ we see
with \eqref{eq-40} that
\begin{equation*}
    \left\{\begin{array}{cl}
      (a) & \xb^a\in [C-\xb^s]\cap\dom\Gamma^* \\
      (b) & \langle \ob,\xb^a\rangle \leq \langle \ob,x\rangle, \forall x\in [C-\xb^s]\cap\dom\Gamma^* \\
      (c) & \lb^a=\gamma'_z(\ott)\,R \\
    \end{array}\right.
\end{equation*}
By Theorem \ref{res-02}-c, this implies that $\ot$ solves
($\widetilde{\mathrm{D}}_{C-\xb^s}$).
\\
It remains to show that $\ot$ also solves
($\widetilde{\mathrm{D}}_{C}$). Thanks to Theorem \ref{res-03}, we
have: $\inf_{x\in C}\langle\ob,x\rangle=\langle
T^*\ob,\lb\rangle=\langle\ott,\lb^a\rangle+\Ib(\lb^s)$ and
$\inf_{x\in C}\langle\ob,x\rangle-I_\gamma(\ott)
=\sup(\dbc)=\inf(\pbc)
=I(\lb^a)+\Ib(\lb^s)=\inf(\overline{\mathrm{P}}_{C-\xb^s})+\Ib(\lb^s)
=\sup(\overline{\mathrm{D}}_{C-\xb^s})+\Ib(\lb^s).$ Therefore
    $\inf_{x\in C}\langle\ob,x\rangle
    =\inf_{x\in C-\xb^s}\langle\ot,x\rangle+\Ib(\lb^s)$
and subtracting $\Ib(\lb^s)$ from $\sup(\dbc)$, we see that $\ot$
which solves ($\widetilde{\mathrm{D}}_{C-\xb^s}$) also solves
($\widetilde{\mathrm{D}}_{C}$). We  complete the proof of the
theorem, taking $\od=\ot.$
\end{proof}

\begin{remark}
Replacing $\lb^s$ with $t\lb^s,$ the same proof shows that $\ot$
solves ($\widetilde{\mathrm{D}}_{C+(t-1)\xb^s}$) for any $t\ge0.$
\end{remark}

From now on,  we denote $\Qd\in \LlsR$  the absolutely continuous
part shared by all the solutions of \eqref{pbc}. Let us introduce
\begin{eqnarray}
    \CC&=&\left\{Q\in\LlsR ;\ T Q:=\IZ\theta\, dQ\in C\right\}\label{eq-15}\\
    \CCb&=&\left\{\ell\in\Ll' ;\ T \ell:=\langle\theta,\ell\rangle\in
    C\right\}\nonumber
    \index{C1@$\CC,$ constraint set}
    \index{C2@$C,$ constraint set}
    \index{C3@$\CCb,$ extended constraint set}
\end{eqnarray}
the constraint sets $T^{-1}C\cap\LlsR$ and $T^{-1}C$ on which $I$
and $\Ib$ are minimized. We have: $\CC=\CCb\cap\LlsR$ and
$I=\Ib+\iota_{\LlsR}.$ Hence, $\inf(\pbc)\le\inf(\pc).$ Clearly,
$\CC\cap\dom I\not=\emptyset\Leftrightarrow\inf(\pc)<\infty$
implies $\CCb\cap\dom
\Ib\not=\emptyset\Leftrightarrow\inf(\pbc)<\infty\Leftrightarrow
C\cap\dom\Gamma^*\not=\emptyset.$
\\
Of course, if $\CC\cap\dom I\not=\emptyset,$ \eqref{pc} admits
nontrivial minimizing sequences. Theorem \ref{res-minseq} below
gives some details about them.

The present paper is concerned with
\begin{equation}\label{pcc}
    \textsl{minimize } I(Q) \textsl{ subject to }  Q\in \CC.
      \tag{$\pcc$}
      \index{C1@$\CC,$ constraint set}
      \index{Primal problems!P3@\eqref{pcc}}
\end{equation}
where  $\CC$ has the form \eqref{eq-15}. But this is not a
restriction as explained in the following remarks.
\begin{remarks}\label{rem-04}\
\begin{enumerate}
    \item Taking $T^*$ to be the identity on $\YYo=\El$ or
    $\Ll$ (being careless with a.e.\! equality, this corresponds
    to $\theta(z)$ to be the Dirac measure $\delta_z$),
we see that \eqref{A-forall} or \eqref{A-exists} is satisfied
respectively. Hence, with $C=\CC,$ \eqref{pc} is \eqref{pcc}.
Consequently, the specific form with $\theta$ and $C$ adds details
to the description of $\CC$ without loss of generality.

    \item With $\theta,$ $\YYo$ and $C$ as in (1),  the
    assumptions on $C=\CC$ are:
    \begin{enumerate}
        \item Under \eqref{A-forall}, \eqref{eq-05} is equivalent to $\CC$ is $\sLE$-closed.\\ Note that if
        $\lmax$ and $\lmaxs$ both
        satisfy the $\Delta_2$-condition, as $\CC$ is convex, this
        is equivalent to $\CC$ being $\|\cdot\|_{\lmaxs}$-closed.
        \item Under \eqref{A-exists}, \eqref{eq-10} is equivalent to $\CC$ is
        $\sigma(\LlsR,\Ll)$-closed.\\
        Note that if $\lmaxs$
        satisfies the $\Delta_2$-condition, as $\CC$ is convex, this
        is equivalent to $\CC$ being $\|\cdot\|_{\lmaxs}$-closed.
    \end{enumerate}
\end{enumerate}
\end{remarks}

We denote $\|\cdot\|_{\lmaxs}$-$\inter(\CC)$ the interior of $\CC$
in $\LlsR$ with respect to the strong topology of $\Lls.$

\begin{theorem}[Minimizing sequences of \eqref{pc}]\label{res-minseq}
Assume that the hypotheses \emph{(A)} hold,
\begin{equation}\label{A4}
    \lim_{t\rightarrow \pm\infty}\gamma_z^*(t)/t=+\infty,\quad
    \hbox{for $R$-almost every }z\in \ZZ
\end{equation}
and $\CC$ which is defined at \eqref{eq-15} is $\sLL$-closed and
satisfies $\CC\cap\dom I\not=\emptyset.$
\\
Let us consider the following additional conditions.
\begin{enumerate}
    \item
        \begin{enumerate}[a-]
            \item There
are finitely many moment constraints, i.e.\! $\XXo=\R^K$ (see
Section \ref{sec:constraints})
            \item $\CC\cap\icordom I\not=\emptyset.$
        \end{enumerate}
    \item
    $\CC$ has a nonempty $\|\cdot\|_{\lmaxs}$-interior.
\end{enumerate}
Under one of these additional conditions (1) or (2), we have
\begin{equation*}
    I(\Qd)\le\inf(\pc)=\inf(\pbc)
\end{equation*}
and any minimizing sequence of \eqref{pc} converges to $\Qd$ with
respect to $\sigma(\LlsR, \El)$ and in variation norm (i.e.\!
strongly in $L_1R$).
\end{theorem}

\begin{proof} This proof relies on results which are stated and
proved in the remainder of the present section. It is shown at
Lemma \ref{res:aa} that any minimizing sequence of \eqref{pc}
converges in the sense of the $\sigma(\LlsR, \El)$-topology to
$\Qd,$ whenever $\inf(\pc)=\inf(\pbc).$ But this equality holds
thanks to Lemma \ref{res:lemdual} and
\begin{itemize}
    \item under condition (1): Lemma \ref{res:lemlsc2}-a;
    \item under condition (2): Corollary \ref{res-02a}-b.
\end{itemize}
Let us have a look at the last inequality. For any
$\lb=\lb^a+\lb^s=\Qd+\lb^s$ minimizer of \eqref{pbc} and any $\seq
Qn$ minimizing sequence of \eqref{pc}, we obtain
\begin{eqnarray*}
  \inf_n I(Q_n) &=& \inf(\pc)=\inf(\pbc)=\Ib(\lb)\\
   &=& I(\Qd)+I^s(\lb^s) \\
   &\geq& I(\Qd)
\end{eqnarray*}
with a strict inequality if $I^s(\lb^s)>0.$
\\
Finally, \eqref{A4} implies that $\lmax$ is finite. As $R$ is
assumed to be bounded, we have $L_\infty\subset \El.$  But, we
also have $\LlsR\subset L_1R$ so that the
$\sigma(\LlsR,\Ll)$-convergence implies the
$\sigma(L_1R,L_\infty)$-convergence of any minimizing sequence.
The strong convergence in $L_1R$ now follows from \cite[Thm
3.7]{BL94}.
\end{proof}

\begin{remarks}\
\begin{enumerate}
    \item As regards condition (1), it is not assumed that $\CC$
    has a nonempty interior.

    \item As regards condition (2):
    \begin{enumerate}[(i)]
        \item Any $\sigma(\LlsR,\Ll)$-closed convex
    set has the form
    $$\CC=\bigcap_{u\in U}\left\{\ell\in\LlsR;\langle u,\ell\rangle \ge
    a_u\right\}$$ for some $U\subset\Ll;$
        \item For $\CC$ to have a nonempty $\|\cdot\|_{\lmaxs}$-interior, it is
enough that $C\cap\XL$ has a nonempty interior in $\XL$ endowed
with the  dual norm $|\cdot|^*_\lmax$ defined at \eqref{eq-46}.
This is a consequence of Lemma \ref{L2}-(a) below.
    \end{enumerate}

 \item The last quantity $I^s(\lb^s)=\inf(\pc)-I(\Qd)$ is
precisely the \emph{gap} of lower $\sigma( \LlsR,
\El)$-semicontinuity of $I:$ $\lim_n Q_n=\Qd$ and
$\inf(\pc)=\liminf_n I(Q_n)\geq I(\lim_n Q_n)=I(\Qd).$
\end{enumerate}
\end{remarks}

\subsection{A preliminary lemma.}
Preliminary results for the proof of Theorem \ref{res-minseq} are
stated below at Lemma \ref{res:aa}. The assumption
\eqref{A-exists} about the critical constraint is $
    T^\ast \YYo\subset \Ll.
$

\begin{lemma}\label{res:aa}
Assume that the hypotheses \emph{(A)} and \eqref{A4} hold, $\CC$
is $\sLL$-closed, $\inf(\pc)<\infty$ and
\begin{equation}\label{eq-14}
    \inf(\pc)=\inf(\pbc)
\end{equation}
Then, any minimizing sequence of \eqref{pc} converges to $\Qd$
with respect to $\sigma( \LlsR, \El).$
\end{lemma}

\proof Let $\seq Qn$ be a minimizing sequence of \eqref{pc}. Since
it is assumed that $\inf(\pc)=\inf(\pbc),$ $\seq Qn$ is also a
minimizing sequence of \eqref{pbc}. By \cite[Lemma 6.2]{Leo08},
$\Ib$ is $\sigma( \Ll', \Ll)$-inf-compact. Hence, we can extract a
$\sigma( \Ll', \Ll)$-convergent subnet $(Q_\alpha)_\alpha$ from
$\seq Qn.$ Let $\ell_*\in\CCb$ denote its limit: we have
$\lim_\alpha \IZ u\,dQ_\alpha=\langle\ell_*,u\rangle$ for all
$u\in \Ll.$ As $\langle\ell_*^s,u\rangle=0,$ for all $u\in \El$
(see Proposition \ref{res-B4}), we obtain: $\lim_\alpha \IZ u\,
dQ_\alpha =\IZ u\, d\ell_*^a$ for all $u\in \El.$ This proves that
$(Q_\alpha)_\alpha$ $\sigma(\El', \El)$-converges to $\ell_*^a.$
As $\El$ is a separable Banach space ($\Ll$ is not separable in
general), the topology $\sigma(\El', \El)=\sigma(\LlsR, \El)$ is
metrizable and we can extract a convergent sub\emph{sequence}
$\seq{\tilde{Q}}k$ from the convergent net $(Q_\alpha)_\alpha.$
Hence, $\seq{\tilde{Q}}k$ is a subsequence of $\seq Qn$ which
$\sigma(\LlsR, \El)$-converges to $\ell_*^a.$
\\
Since $\Ib$ is inf-compact, $\ell_*$ is a minimizer of \eqref{pbc}
and by Theorem \ref{res-07}-a, there is a unique $\Qd$ such for
any minimizing sequence $\seq Qn,$ $\ell_*^a=\Qd.$ Therefore, any
convergent subsequence of $\seq Qn$ converges to $\Qd$ for
$\sigma( \LlsR, \El).$ As any subsequence of a minimizing sequence
is still a minimizing sequence, we have proved that from any
subsequence of $\seq Qn,$ we can extract a sub-subsequence which
converges to $\Qd.$ This proves that $\seq Qn$ converges to $\Qd$
with respect to $\sigma( \LlsR, \El).$
\endproof

\subsection{Sufficient conditions for $\inf(\pc)=\inf(\pbc)$.}
Our aim now is to obtain sufficient conditions for the identity
$\inf(\pc)=\inf(\pbc)$ to hold. Let us rewrite the problems
\eqref{pc} and \eqref{pbc} in order to emphasize their differences
and analogies. Denote
\begin{eqnarray*}
    \Phi_L(u)&=&\Il(u)=\IZ\lambda(u)\,dR,\quad u\in \Ll\\
  \Phi_E(u) &=& \Phi_L(u)+\iota_{\El}(u),\quad u\in \Ll
  \index{Fi@$\Phi_E, \Phi_L$}
\end{eqnarray*}
where $E$ and $L$ refer to $\El$ and $\Ll.$ Their convex
conjugates are
\begin{eqnarray*}
  \Phi_E^*(\ell) &=& \sup_{u\in \El}\{\langle \ell, u\rangle - \Il(u)\},\quad \ell\in \LlsR \\
  \Phi_L^*(\ell) &=& \sup_{u\in \Ll}\{\langle\ell,u\rangle -\Il(u)\},\quad \ell\in \Ll'
  \index{Fis@$\Phi_E^*, \Phi_L^*$}
\end{eqnarray*}
It is shown at \cite[Lemma 6.2]{Leo08} that under the assumption
\eqref{A4}
\begin{equation}\label{eq-08}
\left\{
\begin{array}{rcll}
  I(\ell) &=& \Phi_E^*(\ell-mR),  &\ell\in \El'=\LlsR \\
  \Ib(\ell) &=& \Phi_L^*(\ell-mR), &\ell\in \Ll'=\LlsR\oplus\Ll^s \\
\end{array}
\right.
 \index{I3@$I,$ entropy, see \eqref{eq-16}, \eqref{eq-08}}
    \index{I4@$\Ib,$ extended entropy, see \eqref{III}, \eqref{eq-08}}
\end{equation}
Hence, considering the minimization problems
\begin{equation}\label{eq-55}
    \textsl{minimize } \Phi_E^*(\ell) \textsl{ subject to } \tl \in C_o,\quad \ell\in \LlsR
    \tag{$\pE$}
    \index{Primal problems!P4@\eqref{eq-55}}
\end{equation}
and
\begin{equation}\label{eq-56}
    \textsl{minimize } \Phi_L^*(\ell) \textsl{ subject to }\tl\in C_o,\quad \ell\in \Ll'
    \tag{$\pL$}
    \index{Primal problems!P5@\eqref{eq-56}}
\end{equation}
with $C_o=C-\langle\theta,mR\rangle,$ we see that $\ell_*$ is a
solution of \eqref{pc} [resp. \eqref{pbc}] if and only if
$\ell_*-mR$ is a solution of \PE\ [resp. \PL]. It is enough to
prove
\begin{equation}\label{eq-07}
    \inf(\pE)=\inf(\pL)
\end{equation}
 to get $\inf(\pc)=\inf(\pbc).$

\par\medskip
\noindent\textbf{Basic facts about convex duality.} The proof of
\eqref{eq-07} will rely on standard convex duality considerations.
Let us recall some related facts, as developed in \cite{Roc74}.
\\
Let $A$ and $X$ be two vector spaces, $h: A\to[-\infty,+\infty]$ a
convex function, $T:A\to X$ a linear operator and $C$ a convex
subset of $X.$ The primal problem to be considered is the
following convex minimization problem
\begin{equation}\label{eq-57}
    \textsl{minimize } h(a) \textsl{ subject to } Ta\in C,\quad a\in A
    \tag{$\mathcal{P}$}
    \index{Primal problems!P6@\eqref{eq-57}}
\end{equation}
The primal value-function corresponding to the Fenchel
perturbation $F(a,x)=h(a)+\iota_C(Ta+x),$ $a\in A, x\in X$ is
$\varphi(x)=\inf_{a\in A}F(a,x),$ i.e.
\[
\varphi(x)=\inf\{h(a); a\in A, Ta\in C-x\}\quad x\in X.
\]
Denote $A^*$ the algebraic dual space of $A.$ The convex conjugate
of $h$ with respect to the dual pairing $\langle A,A^*\rangle$ is
\begin{equation*}
    h^*(\nu)=\sup_{a\in A}\{\langle \nu,a\rangle -h(a)\},\quad \nu\in A^*.
\end{equation*}
Let $Y$ be a vector space in dual pairing with $X.$ The adjoint
operator $T^*: Y\to A^*$ of $T$ is defined for all $y\in Y$  by
\begin{equation*}
   \langle T^*y,a\rangle_{A^*,A}=\langle y,Ta\rangle_{Y,X},\quad \forall a\in A
\end{equation*}
The dual problem associated with ($\mathcal{P}$) is
\begin{equation}\label{eq-58}
    \textsl{maximize } \inf_{x\in C}\langle y,x\rangle - h^*(T^\ast y),\quad y\in Y
    \tag{$\mathcal{D}$}
    \index{Dual problems!D4@\eqref{eq-58}}
\end{equation}
Let $U$ be some subspace of $A^*.$  The dual value-function is
\begin{equation*}
    \psi(u)=\sup_{y\in Y} \left\{\inf_{x\in C}\yx-h^*(T^*y+u)\right\},\quad
    u\in U
\end{equation*}
We say that $\langle X,Y\rangle$ is a \emph{topological dual
pairing} if $X$ and $Y$ are topological vector spaces and their
topological dual spaces $X'$ and $Y'$ satisfy $X'=Y$ and $Y'=X$ up
to some isomorphisms.

\begin{theorem}[Criteria for the dual equality]\label{res-01}
We assume that  $A,U,X$ and $Y$ are locally convex Hausdorff
topological vector spaces such that $\langle A,U\rangle$ and
$\langle X,Y\rangle$ are topological dual pairings. For the dual
equality
$$\inf(\mathcal{P})=\sup(\mathcal{D})$$ to hold, it is enough that
\begin{enumerate}
    \item
    \begin{enumerate}
        \item $h$ is a convex function and $C$ is a convex subset of
        $X,$
        \item $\varphi$  is \lsc\ at $0\in X$ and
        \item $\sup(\mathcal{D}) >-\infty$
    \end{enumerate}
\end{enumerate}
or
\begin{enumerate}
    \item[(2)]
    \begin{enumerate}
        \item $h$ is a convex function and $C$ is a closed convex subset of
        $X,$
        \item $\psi$ is \usc\ at $0\in U$ and
        \item $\inf(\mathcal{P})<+\infty.$
    \end{enumerate}
\end{enumerate}
\end{theorem}

\begin{remarks} About the space $U.$
\begin{enumerate}[(a)]
    \item As regards Criterion (1), the space $U$ is unnecessary.
    \item As regards Criterion (2), it is not assumed that $T^*Y\subset
    U.$
\end{enumerate}
\end{remarks}

\par\medskip
\noindent\textbf{The continuity of $T$ and $T^*$.} The following
Lemma \ref{L2} will be useful for the proof of \eqref{eq-07}. We
go back to our usual notation. The Luxemburg norm on $\Ll$ is
$\NF$ and its dual norm is
$$
\|\ell\|_{\lmax}^*:= \sup_{u, \|u\|_{\lmax}\leq 1}|\ul|,\quad
\ell\in\Ll'
$$
Let us define
\begin{equation*}
    |y|_\lmax= \|\langle y,\theta\rangle\|_\lmax,\quad y\in\YYo
\end{equation*}
Under the assumption (A$_\theta$), $ |\cdot|_\lmax$ is a norm on
$\YYo.$ The dual space of $(\YYo,|\cdot|_\lmax)$ is $\XX$ and the
corresponding dual  norm is
\begin{equation}\label{eq-46}
    |x|_\lmax^*:= \sup_{y, |y|_\lmax\leq 1}|\yx|, \quad x\in\XX
\end{equation}

\begin{lemma}\label{L2}
Let us assume \emph{(A$_\theta$)}.
\begin{enumerate}[(a)]
    \item $T: \Ll'\to\XX$ is  $\NF^*$-$|\cdot|_\lmax^*$-continuous
    \item $T^\ast \YY\subset\Ll$ and $T^\ast: \YY\to\Ll$ is $\sigma(\YY,\XX)$-$\sigma(\Ll,\Ll')$-continuous
    \item $T: \Ll'\to\XX$ is $\sigma(\Ll',\Ll)$-$\sigma(\XX,\YY)$-continuous
\end{enumerate}
\end{lemma}
\begin{proof}
See \cite[Section 4]{p-Leo07a}.
\end{proof}

\par\smallskip
\noindent\textbf{Back to our problem.} Let us particularize this
framework for the problems \PE\ and \PL. Assuming that $m\equiv
0,$ we see that \eqref{pc}=\PE, \eqref{pbc}=\PL,
$I=\Ils=\Phi_E^*,$ $\Ib=\Ilsb=\Phi_L^*,$ $\gamma=\lambda,$ $C=C_o$
and so on. \emph{This simplifying requirement will be assumed
during the proof without loss of generality, see the proof of
\cite[Theorem 4.2]{Leo08}.}

Let us first apply the criterion (1) of Theorem \ref{res-01}.

\par\noindent Problem \PE\ is obtained with $A=\LlsR,$ $X=\XL,$ $Y=\YL$ equipped with the weak topologies
$\sigma(\XL,\YL)$ and $\sigma(\YL,\XL),$ and $h=\Phi_E^*=I.$ The
corresponding primal value-function is
\begin{equation*}
  \varphi_E(x) = \inf\left\{I(Q); \IZ\theta\,dQ\in C-x, Q\in \LlsR\right\},\quad x\in
  \XL
\end{equation*}
Under the underlying assumption \eqref{A-exists}, with Lemma
\ref{L2}-(b) we have:
\begin{equation}\label{eq-30}
    T^\ast \YL\subset \Ll.
\end{equation}
Hence, we only need to compute $h^*$ on $\Ll\subset
[\LlsR]^*=A^*.$ For each $u\in \Ll,$ $h^*(u)=\sup_{f\in \Lls}\{\IZ
uf\,dR-\IZ\lambda^*(f)\,dR\}$ and it is proved in \cite{Roc68}
that
\begin{equation}\label{eq-11}
    h^*(u)=\IZ\lambda(u)\,dR,\quad u\in\Ll
\end{equation}
Therefore, the dual problem associated to \PE\ is
\begin{equation}\label{eq-59}
    \textsl{maximize } \inf_{x\in C}\langle y,x\rangle-\IZ\lambda(T^\ast y)\,dR,\quad y\in \YL
    \tag{$\dE$}
    \index{Dual problems!D5@\eqref{eq-59}}
\end{equation}

Let us go on with \PL. Take $A=\Ll',$ $X=\XL,$ $Y=\YL$ equipped
with the weak topologies $\sigma(\XL,\YL)$ and $\sigma(\YL,\XL),$
and $h=\Phi_L^*=\Ib.$ The function $h^*$ in restriction to $\Ll$
is still given by \eqref{eq-11} since
$u\in\Ll\mapsto\IZ\lambda(u)\,dR$ is closed convex (Fatou's
lemma). The primal value-function is
\begin{equation*}
  \varphi_L(x) = \inf\{\Ib(\ell); \tl\in C-x, \ell\in \Ll'\},\quad x\in
  \XL
\end{equation*}
and the dual problem associated to \PL\ is
\begin{equation*}
 (\dL)=(\dE).
\end{equation*}

\begin{lemma}\label{res:lemlsc1}
    Suppose that $T^*\YL\subset\Ll$  and $C\cap\XL$ is
    $\sigma(\XL,\YL)$-closed, then $\varphi_L$
    is $\sigma(\XL,\YL)$-\lsc.
\end{lemma}
\proof Defining $\tilde\varphi(x):=\varphi_L(-x)$ and
$\Jb(x):=\inf\{\Ib(\ell);\ell\in\Ll': \tl=x\},$ $x\in\XL,$ we
obtain that $\tilde\varphi$ is the inf-convolution of $\Jb$ and
the convex indicator of $-C:$ $\iota_{-C}.$ That is
$\tilde\varphi(x)=(\Jb\Box
\iota_{-C})(x)=\inf\{\Jb(y)+\iota_{-C}(z); y,z, y+z=x\}.$
\\
As already seen, $\Ib$ is $\sigma(\Ll',\Ll)$-inf-compact and $T$
is $\sigma(\Ll',\Ll)$-$\sigma(\XL,\YL)$-continuous, see Lemma
\ref{L2}-(c). It follows that $\Jb$ is
$\sigma(\XL,\YL)$-inf-compact. As $C\cap\XL$ is assumed to be
$\sigma(\XL,\YL)$-closed, $\iota_{-C}$ is \lsc. Finally, being the
inf-convolution of an inf-compact function and a \lsc\ function,
$\tilde\varphi$ is \lsc, and so is $\varphi_L.$
\endproof
As $\Ib$ and $C$ are assumed to be convex and $\sup(\dL)\ge
\inf_{x\in C}\langle 0,x\rangle-\IZ\lambda(T^\ast
0)\,dR=0>-\infty,$ this lemma allows us to apply Criterion (1) of
Theorem \ref{res-01} to obtain
\begin{equation}\label{eq-13}
    \inf(\pL)=\sup(\dL)
\end{equation}
Since $\Ib$ and $I$ match on $\LlsR,$ we have
$\inf(\pL)\leq\inf(\pE).$ Putting together these considerations
gives us
\begin{equation*}
\sup(\dE)=\sup(\dL)=\inf(\pL)\leq\inf (\pE).
\end{equation*}
Since the desired equality \eqref{eq-14} is equivalent to
$\inf(\pL)=\inf(\pE),$ we have proved
\begin{lemma}\label{res:lemdual}
The equality \eqref{eq-14} holds if and only if we have the dual
equality
\begin{equation}\label{eq-12}
    \inf(\pE)=\sup(\dE).
\end{equation}
This happens if and only if $\varphi_E$ is $\sigma(\XL,\YL)$-\lsc\
at $x=0.$
\end{lemma}

Let us now give a couple of simple criteria for this property to
be realized.
\begin{lemma}\label{res:lemlsc2}\quad
\begin{enumerate}[(a)]
    \item
    Suppose that there are finitely many constraints (i.e.\! $\XXo$
    is finite dimensional) and $\CC\cap\icordom I\not=\emptyset,$
    then $\varphi_E$ is continuous at $0.$
    \item
    Suppose that \eqref{A-forall} is satisfied and $C$ is
    $\sigma(\XE,\YE)$-closed, then $\varphi_E$
    is $\sigma(\XL,\YL)$-\lsc.
\end{enumerate}
\end{lemma}
\proof \boulette{(a)} To get (a), remark that a convex function on
a finite dimensional space is \lsc\ on the intrinsic core of its
effective domain. By Lemma \ref{L2}-(c), $T$ is $\sLL$-continuous
and the assumption $\CC\cap\icordom I\not=\emptyset$ implies that
$0$ belongs to $\icordom\varphi_E.$

\Boulette{(b)} It is similar to the proof of Lemma
\ref{res:lemlsc1}. The assumption \eqref{A-forall} insures that
$T$ is $\sLE$-$\sigma(\XE,\YE)$-continuous.
\endproof

It follows from Lemma \ref{res:lemlsc2}-b, Lemma
\ref{res:lemdual}, the remark at \eqref{eq-07} and Lemma
\ref{res:aa} that under the good constraint  assumption
(A$^\forall_\theta$), if $C\cap\dom\Gamma^*\not=\emptyset,$ then
any minimizing sequence of \eqref{pc} converges with respect to
the topology $\sigma( \LlsR, \El)$ to the unique solution $\Qh$ of
\eqref{pc}. This is Theorem \ref{res-02}-b.

\par\medskip\noindent\textbf{Using Criterion (2).}\ Up to now, we
only used Criterion (1) of Theorem \ref{res-01}. In the following
lines, we are going to use Criterion (2) to prove \eqref{eq-12}
under additional assumptions.

Let us go back to Problem \PE. It is still assumed without loss of
generality that $m=0$ and $\gamma=\lambda.$ We introduce a space
$U$ and a dual value-function $\psi$ on $U.$ The framework of
Theorem \ref{res-01} is preserved when taking $X=\XL,$ $Y=\YL$
with the weak topologies $\sigma(\XL,\YL)$ and $\sigma(\YL,\XL),$
$h=I$ on $A=\LlsR$ as before and adding the following topological
pairing $\langle A,U\rangle.$ We endow $A=\LlsR$ with the topology
$\sigma(\LlsR,\Ll)$ and take
$$U=A'=(\LlsR)'\simeq\Ll$$ with the topology $\sigma(\Ll,\Lls).$
By \eqref{eq-11}:  $h^*=\Il,$ this leads to the dual
value-function
\begin{equation*}
    \psi(u)=\sup_{y\in \YL} \left\{\inf_{x\in C}\yx-\Il(T^*y+u)\right\},\quad
    u\in \Ll
\end{equation*}
To apply Criterion (2), let us establish the following

\begin{lemma}\label{res-03a}
For $\psi$ to be $\sigma(\Ll,\Lls)$-\usc, it is enough that
\begin{enumerate}[(a)]
    \item   $T^*\YL$ is a $\sigma(\Ll,\Lls)$-closed subspace
    of $\Ll$ and
    \item the interior of $\CC$ in $\LlsR$ with respect to $\|\cdot\|_{\lmaxs}$ is nonempty.
\end{enumerate}
\end{lemma}

\begin{proof}
During this proof, unless specified the topology on $\Ll$ is
$\sigma(\Ll,\Lls).$ For all $u\in\Ll,$
\begin{eqnarray*}
   -\psi(u)
   &=& \inf_{y\in\YL}\Big\{\sup\{\langle-T^*y,\ell\rangle;\ell\in\LlsR:T\ell\in C\}+\Il(T^*y+u)\Big\} \\
    &=& \inf_{v\in V}\Big\{\sup\{\langle-v,\ell\rangle;\ell\in\CC\}+\Il(v+u)\Big\}\\
    &=& \Il\Box G (u)
\end{eqnarray*}
where  $V=T^*\YL$ and $\Il\Box
G(u)=\inf_{v\in\Ll}\{G(v)+\Il(u-v)\}$ is the inf-convolution of
$\Il$ and
    $
    G(u)=\iota_\CC^*(u)+\iota_{V}(u), u\in\Ll.
    $
Let us show that under the assumption (a),
\begin{equation*}
    G=\iota_\CC^*
\end{equation*}
 As $V$ is assumed to be closed, we have
$\iota_V=\iota^*_{V^\bot}$ with $V^\bot=\{k\in\LlsR; \langle
v,k\rangle=0, \forall v\in V\}.$ This gives for each $u\in\Lls',$
$G(u)=\iota^*_\CC(u)+\iota^*_{V^\bot}(u)=\sup_{\ell\in\CC}\langle
u,\ell\rangle+\sup_{k\in V^\bot}\langle u,k\rangle=\sup\{\langle
u,\ell+k\rangle;\ell\in\CC, k\in
V^\bot\}=\iota^*_{\CC+V^\bot}(u)=\iota^*_{\CC}(u),$ where the last
equality holds since $\CC+V^\bot=\CC,$ note that
$\mathrm{ker\,}T=V^\bot.$
\\
As convex conjugates, $\iota_\CC^*$ and $\Il=h^*$ are closed
convex functions.
\\
Since for all $u,v\in\Ll,$
\begin{eqnarray*}
  \iota_\CC^*(v)+\Il(u-v)
  &=& \iota_{\CC-k}^*(v)-\langle v,k\rangle+\Il(u-v) \\
  &=& \iota_{\CC-k}^*(v)+\langle u-v,k\rangle+\Il(u-v)-\langle
  u,k\rangle,
\end{eqnarray*}
we have
$$
-\psi+\langle \cdot,k\rangle=\iota^*_{\CC-k}\Box
(\Il-\langle\cdot,k\rangle)
$$
But, by the assumption (b) there exists some $k$ in $\LlsR$ such
that $0\in\inter (\CC-k).$ It follows that $\iota^*_{\CC-k}$ is
inf-compact. Finally, $-\psi+\langle \cdot,k\rangle$ is \lsc,
being the inf-convolution of a \lsc\ and an inf-compact functions.
\end{proof}

\begin{corollary}\label{res-02a} Assume that the hypotheses \emph{(A)} and \eqref{A4}
hold.
\begin{enumerate}[(a)]
    \item Assume in addition that $C$ is $\sigma(\XL,\YL)$-closed convex, $T^*\YL$ is a
$\sigma(\Ll,\Lls)$-closed subspace of $\Ll, $ $\CC$ has a nonempty
$\|\cdot\|_{\lmaxs}$-interior and $\inf(\pE)<\infty.$ Then,
\eqref{eq-12} is satisfied.
    \item In particular, if $\CC$ is $\sigma(\LlsR,\Ll)$-closed convex set with a nonempty $\|\cdot\|_{\lmaxs}$-interior and
$\inf(\pE)<\infty,$ then \eqref{eq-12} is satisfied.
\end{enumerate}
\end{corollary}

\begin{proof}
    \boulette{(a)}
Apply the criterion (2) of Theorem \ref{res-01} with  Lemma
\ref{res-03a}.

    \Boulette{(b)} This is (a) with $\YL=\Ll$ and
    $T^*=\mathrm{Id},$ taking advantage of Remarks \ref{rem-04}.
\end{proof}

\section{Entropic projections} \label{sec-entproj}

The results of the preceding sections are translated in terms of
entropic projections. We consider the problem \eqref{pcc} where
$\CC$ is a convex subset of $\MZ.$

\subsection{Generalized entropic projections.}
Let us start recalling an interesting result of Csisz\'ar.

\begin{definitions} Let $Q$ and  $\seq Qn$ in $\MZ$ be absolutely continuous with
respect to $R.$
\begin{enumerate}
    \item  One says that $\seq Qn$ converges  in
$R$-measure to $Q$ if $\frac{dQ_n}{dR}$ converges in $R$-measure
to $\frac{dQ}{dR}.$

  \item  One says that $\seq Qn$ converges $\sLE$ to $Q$ in
$\LlsR$ if $\frac{dQ_n}{dR}$ $\sLE$-converges to $\frac{dQ}{dR}.$

    \item One says that $\seq Qn$ converges in
variation to $Q$ if $\frac{dQ_n}{dR}$ converges to $\frac{dQ}{dR}$
in $L_1.$
\end{enumerate}
\end{definitions}

\begin{definition}[Generalized entropic projection]\label{defCsiszar}
\cite[Csisz\'ar]{Csi84}. Suppose that $\CC\cap\dom
I\not=\emptyset$ and that any minimizing sequence of the problem
\eqref{pcc} converges in variation to some
$Q_*\in\MZ.\index{Qs@$Q_*,$ generalized  $I$-projection}$ This
$Q_*$ is called the \emph{generalized  $I$-projection} of $mR$ on
$\CC$ with respect to $I.$ It may not belong to $\CC.$ In case
$Q_*$ is in $\CC,$ it is called the \emph{$I$-projection} of $mR$
on $\CC.$
    \index{genIproj@generalized $I$-projection, see Definition \ref{defCsiszar}}
    \index{Iproj@$I$-projection, see Definition \ref{defCsiszar}}
     \index{genentproj@generalized entropic projection, see Definition \ref{defCsiszar}}
    \index{entproj@entropic projection, see Definition \ref{defCsiszar}}
\end{definition}

\begin{theorem}[Csisz\'ar, \cite{Csi95}]\label{res:Csis}
Suppose that \emph{(A$_R$)} and \emph{(A$_{\gamma^*}$)} are
satisfied. Then, $mR$ has a generalized $I$-projection on any
convex subset $\CC$ of $\MZ$ such that $\CC\cap\dom
I\not=\emptyset.$
\end{theorem}

In \cite{Csi95} $\gamma^*$ doesn't depend on the variable $z,$ but
the proof remains unchanged with a $z$-dependence.

Csisz\'ar's proof of Theorem \ref{res:Csis} is based on a
parallelogram identity which allows to show that any minimizing
sequence is a Cauchy sequence. This result is general but it
doesn't tell much about the nature of $Q_*.$ Let us give some
details on the generalized entropic projections in specific
situations.

As a direct consequence of Theorem \ref{res-02}, if the
constraints are good, the generalized entropic projection is the
entropic projection.

\begin{proposition}\label{res:Iproj}
   Suppose that  \emph{(A)}, \emph{(A$_{\theta}^\forall$)} and \eqref{A4} hold,
   $C$ is convex, $\CC$ given at \eqref{eq-15} is $\sLE$-closed and $\CC\cap\dom
   I\not=\emptyset.$ Then, the  $I$-projection $Q_*$ exists and is equal to
$$Q_*=\Qh\in\CC$$
   where $\Qh$ is the minimizer of \eqref{pc} which is described at Theorem \ref{res-02}.
   \\
    Moreover, any minimizing sequence  $\sLL$-converges to $Q_*.$
   \end{proposition}

\proof This is an easy corollary of Theorem \ref{res-02}.
\endproof

As a direct corollary of Proposition \ref{res:Iproj} and Remarks
\ref{rem-04}, we obtain the following result which is essentially
\cite[Thm 3-(iii)]{Csi95}.
\begin{corollary}
Suppose that  \emph{(A)} and \eqref{A4} hold. Let $\CC$ be any
$\sLE$-closed convex set such that $\CC\cap\dom
   I\not=\emptyset.$ Then, the  $I$-projection of $mR$ on $\CC$ exists.
\end{corollary}

\begin{theorem}\label{res-05}
Suppose that  \emph{(A)} and \eqref{A4} hold,
  $\CC$ given at \eqref{eq-15} is $\sLL$-closed and $\CC\cap\dom
   I\not=\emptyset.$
Let us consider the additional conditions:
\begin{enumerate}
    \item
        \begin{enumerate}[a-]
            \item  There are finitely many moment constraints, i.e.\! $\XXo=\R^K$ (see
Section \ref{sec:constraints})
            \item $\CC\cap\icordom I\not=\emptyset;$
        \end{enumerate}
    \item
    $\CC$ is a $\sigma(\LlsR,\Ll)$-closed convex set with a nonempty $\|\cdot\|_{\lmaxs}$-interior.
\end{enumerate}
Then, under one of the conditions (1) or (2), the generalized
 $I$-projection $Q_*$ of $mR$ on $\CC$ is
$$
Q_*=\Qd
$$
the absolutely continuous component described at \eqref{eq-18} and
$I(Q_*)\le\inf_\CC I.$
\end{theorem}

\begin{proof}
This is a direct consequence of Theorem \ref{res-minseq}.
\end{proof}

\subsection{The special case of relative entropy}\label{sec:entropy2}
The relative entropy $I(P|R)$  and its extension $\Ib(\ell| R)$
are described at Section \ref{sec-expl}. The minimization problem
is
\begin{equation}\label{eq-21}
     \textsl{minimize } I(P|R) \textsl{ subject to }
  \IZ\theta\,dP\in C,\quad P\in \PZ
\end{equation}
and its extension is
\begin{equation}\label{eq-22}
     \textsl{minimize } \Ib(\ell| R) \textsl{ subject to }
  \langle\theta,\ell\rangle\in C,\quad \ell\in \OZ
\end{equation}
Recall that $\OZ$ is defined at \eqref{eq-44}. We introduce the
Cram\'er transform of the image law of $R$ by
 $\theta$ on $\Xx:$
\begin{equation}\label{cramer}
\Xi(x)=\sup_{y\in\Yy}\left\{\langle y,x\rangle-\log\IZ e^{\langle
y,\theta\rangle}\,dR\right\}\in [0,\infty],\quad x\in\Xx
    \index{Xi@$\Xi,$ see \eqref{cramer}}
\end{equation}

\begin{proposition}[Relative entropy subject to good constraints]\label{res:entrop1}

Let us assume that $\theta$ satisfies the``good constraint"
assumption
\begin{equation}\label{gc}
\forall y\in\Yy, \IZ e^{\langle y,\theta(z)\rangle}\,R(dz) <\infty
\end{equation}
and that $C\cap\XE$ is a $\sigma(\XE,\YE)$-closed convex subset of
$\XE.$

\begin{enumerate}[(a)]
    \item The following dual equality holds:
\begin{equation*}
\inf\{I(P| R); \langle\theta,P\rangle\in C, P\in \PZ\}
=\sup_{y\in\Yy}\left\{\inf_{x\in C}\langle y,x\rangle -\log \IZ
e^{\langle y,\theta\rangle}\,dR\right\}\in [0,\infty]
\end{equation*}

    \item Suppose that in addition $C\cap\dom\Xi\not=\emptyset.$
Then, the minimization problem \eqref{eq-21} has a unique solution
$\widehat{P}$ in $\PZ$, $\widehat{P}$ is the entropic projection
of $R$ on $\CC=\{P\in\PZ, \IZ\theta\,dP\in C\}.$

    \item
Suppose that in addition, $C\cap \icordom\Xi\not=\emptyset,$
  then there exists some linear form $\ot$ on $\Xx$ such that $\ott$ is measurable and
\begin{equation}\label{eq-33}
    \left\{
        \begin{array}{l}
         \xh:= \IZ\theta\,d\widehat{P}\in C\cap\dom\Xi \\
         \langle \ot,\xh\rangle \leq \langle \ot,x\rangle,
\forall x\in C\cap\dom\Xi \\
         \widehat{P}(dz)=\exp\left(\langle
\ot,\theta(z)\rangle-\log\IZ e^{\langle
\ot,\theta\rangle}\,dR\right)\,R(dz). \\
        \end{array}
    \right.
    \index{Ph@$\widehat{P},$ see \eqref{eq-33}}
\end{equation}
In this situation, $\xh$ minimizes $\Xi$ on $C,$
$I(\widehat{P}\mid R)=\Xi(\xh)$ and
\begin{equation}\label{eq-31}
 \xh=\IZ \theta(z)\exp\left(\langle
\ot,\theta(z)\rangle-\log\IZ e^{\langle
\ot,\theta\rangle}\,dR\right)\,R(dz)
\end{equation}
in the weak sense. Moreover, $\ot$ satisfies \eqref{eq-33} if and
only if it is the solution to
\begin{equation}\label{eq-35}
     \textsl{maximize } \inf_{x\in \CX}\langle\omega,x\rangle
     -\log\IZ e^{\langle \omega,\theta\rangle}\,dR,
   \quad \omega\in\YYt
\end{equation}
where $\YYt$ is the cone of all measurable functions $u$ such that
$u_-$ is in the $\|\cdot\|_1$-closure of
$\langle\YYo,\theta\rangle_-$ and $u_+$ is in the
$\sigma(L_\tau,L_{\tau^*})$-closure of
$\langle\YYo,\theta\rangle_+.$
\end{enumerate}
\end{proposition}

\begin{proposition}[Relative entropy subject to critical
constraints]\label{res:entrop2}

Let us assume that $\theta$ satisfies the ``critical constraint"
assumption
\begin{equation}\label{bc}
\forall y\in\Yy, \exists\alpha>0, \IZ e^{\alpha |\langle
y,\theta(z)\rangle|}\,R(dz) <\infty
\end{equation}
and that $C\cap\XL$ is a $\sigma(\XL,\YL)$-closed convex subset of
$\XL.$

\begin{enumerate}[(a)]
    \item The following dual equality holds:
\begin{equation*}
\inf\{\Ib(\ell\mid R); \langle\theta,\ell\rangle\in C, \ell\in
\OZ\} =\sup_{y\in\Yy}\left\{\inf_{x\in C}\langle y,x\rangle -\log
\IZ e^{\langle y,\theta\rangle}\,dR\right\}\in [0,\infty]
\end{equation*}

    \item Suppose that in addition $C\cap\dom\Xi\not=\emptyset.$ Then, the
minimization problem \eqref{eq-22} is attained in $\OZ:$ the set
of minimizers is nonempty, convex and
$\sigma(L_{\tau}',L_{\tau})$-compact. Moreover,
 all the minimizers share the same unique absolutely continuous
part $P_\diamond\in \PZ\cap L_{\tau^*}R$ which is the generalized
entropic projection of $R$ on $\CC.$

    \item Suppose that in addition, one of the following conditions
\begin{enumerate}[(1)]
    \item
          $\XXo=\R^K$ and  $\CC\cap\icordom
          I\not=\emptyset$ or
    \item
         $\|\cdot\|_{\tau^*}$-$\inter(\CC)\not=\emptyset.$
\end{enumerate}
is satisfied. Then, there exists a linear form $\ot$ on $\XXo$
such that $\ott$ is measurable, $\IZ e^{\ott}\,dR<\infty$ and
\begin{equation}\label{eq-34}
P_\diamond(dz)=\exp\left(\langle \ot,\theta(z)\rangle -\log\IZ
e^{\ott}\,dR\right) R(dz).
    \index{Pd@$P_\diamond,$ see \eqref{eq-34}}
\end{equation}
Moreover, $\ot$ is the solution to \eqref{eq-35}.
\end{enumerate}
\end{proposition}

In Proposition \ref{res:entrop1}, $\xh$ is the dominating point in
the sense of Ney (see Definition \ref{dompointNey}) of $C$ with
respect to $\Xi.$ The representation of $\xh$ has already been
obtained for $C$ with a nonempty topological interior by Ney
\cite{Ney84} and  Einmahl and Kuelbs \cite{EK96}, respectively in
$\mathbb{R}^d$ and in a Banach space setting. The representation
of the generalized projection $P_\diamond$ is obtained with a very
different proof by Csisz\'ar \cite{Csi75} and (\cite{Csi84}, Thm
3). Proposition \ref{res:entrop2} also extends corresponding
results of Kuelbs \cite{Kue00} which are obtained in a Banach
space setting.

\proof[Proof of Propositions \ref{res:entrop1} and
\ref{res:entrop2}] They are direct consequences of Theorems
\ref{res-02}, \ref{res-03}, Proposition \ref{res:Iproj}, Theorem
\ref{res-05} and Lemma \ref{astuce} below.
\endproof

The following lemma allows to apply the results of the present
paper with $\gamma(s)=e^s-1$ and the extended constraint
$\langle(\1,\theta),\ell\rangle\in \{1\}\times C:$ the first
component of the constraint insures the unit mass
$\langle\1,\ell\rangle=1,$ to obtain results in terms of
log-Laplace transform.

\begin{lemma}\label{astuce}
For all $x\in\Xx,$
\[\sup_{y\in\Yy}\left\{\langle y,x\rangle -\log\IZ
e^{\langle y,\theta\rangle}\,dR\right\}=
\sup_{\tilde{y}\in\mathbb{R}\times\Yy}\left\{\langle\tilde{y},(1,x)\rangle
-\IZ (e^{\langle \tilde{y},(1,\theta)\rangle}-1)\,dR\right\}\in
(-\infty,+\infty].\]
\end{lemma}
\proof Using the identity: $-\log b=\sup_a\{a+1-be^a\},$ we get:
\begin{eqnarray*}
  \sup_{y\in\Yy}\left\{\langle y,x\rangle -\log\IZ e^{\langle
y,\theta\rangle}\,dR\right\}
&=&\sup_{a\in\mathbb{R},y\in\Yy}\left\{\langle y,x\rangle +a+1
-e^a\IZ e^{\langle y,\theta\rangle}\,dR\right\}\\
&=&\sup_{a\in\mathbb{R},y\in\Yy}\left\{\langle(a,y),(1,x)\rangle
-\IZ
e^{\langle y,\theta\rangle+a}\,dR+1\right\}\\
&=&\sup_{\tilde{y}\in\mathbb{R}\times\Yy}\left\{\langle\tilde{y},(1,x)\rangle
-\IZ (e^{\langle \tilde{y},(1,\theta)\rangle}-1)\,dR\right\}
\end{eqnarray*}
which is the desired result.
\endproof

Formula \eqref{cramer} states that $\Xi=\Lambda^*$ where
\begin{equation*}
    \Lambda(y)=\log\IZ e^{\langle y,\theta\rangle}\,dR,\quad
    y\in\YYo
\end{equation*}
is the log-Laplace transform of the image of $R$ by the mapping
$\theta.$ On the other hand, Lemma \ref{astuce} states that
$\Lambda^*(x)=\Gamma^*(1,x).$ Hence, when working with the
relative entropy, one can switch from $\Gamma$ to $\Lambda.$

\begin{example}\label{exCsi1} \textit{Csisz\'ar's example}. Comparing Proposition
\ref{res:Iproj} with Theorem \ref{res-05}, one may wonder if the
$\sLE$-closedness of the convex set $\CC$ is critical for the
existence of an entropic projection. The answer is affirmative. In
\cite[Example 3.2]{Csi84}, Csisz\'ar gives an interesting example
where the generalized entropic projection can be explicitly
computed in a situation where \eqref{pc} is not attained, see also
\cite[Exercise 7.3.11]{DZ}. This example is the following one.
\\
Take the probability measure on $\ZZ=[0,\infty)$ defined by
$R(dz)= \frac 1{a_0}\frac{e^{-z}}{1+z^3}\,dz$ where $a_0$ is the
normalizing constant, $I$ the relative entropy with respect to
$R,$  $\theta: z\in [0,\infty)\mapsto z\in\XX=\R$ and
$C=[c,\infty).$ This gives $\CC=\{Q\in L_{\tau^*}R;
\int_{[0,\infty)}z\,Q(dz)\ge c, Q([0,\infty))=1\}.$
\\
From the point of view of the present article, the main point  is
that $\theta$ is in $L_\tau(R)$ but not in $E_\tau(R).$ Indeed,
the log-Laplace transform is
\begin{equation}\label{eq-62}
    \Lambda(y)=\log\int_{[0,\infty)}e^{(y-1)z}\frac{dz}{a_0(1+z^3)},\quad y\in\R
\end{equation}
 whose effective
domain is $(-\infty,1]$ which admits 0 as an interior point but is
not the whole line. A graphic representation of $\Lambda$ is drawn
at Figure \ref{fig3} below, see Example \ref{exCsi2}.

By Theorem \ref{res:entrop2}, the generalized projection of $R$ on
$\CC$ exists and is equal to $P_y=\frac
1{a_y}\frac{e^{(y-1)z}}{1+z^3}\,dz$ for some real $y\le1,$ where
$a_y=a_0e^{\Lambda(y)}$ is the normalizing constant. If it belongs
to $\CC,$ then $\Lambda'(y)=\IZ z\,P_y(dz)\ge c.$ But $\Lambda$ is
convex so that $\Lambda'$ is increasing and $\sup_{y\le1}\IZ
z\,P_y(dz)=\IZ z\,P_1(dz)=\frac 1{a_1}\IZ
\frac{z}{1+z^3}\,dz:=x_*<\infty.$ Therefore, for any $c>x_*,$
there are no entropic projection but only a generalized one which
is $P_1(dz)=\frac{dz}{a_1(1+z^3)}.$

A detailed analysis of this example in terms of singular component
is done by L\'eonard and Najim  \cite[Proposition 3.9]{LeoN02}.
This example corresponds to a $\sigma(\LlsR,\Ll)$-closed convex
set $\CC$ such that no entropic projection exists but only a
generalized one.
\\
More details about this example are given below at Examples
\ref{exCsi2} and \ref{exCsi3}.
\end{example}

\section{Dominating points}\label{sec:dompoints}

Following Ney \cite{Ney83,Ney84}, let us introduce the following
definition. A point $\xh\in\Xx$ sharing the properties
\eqref{eq-36}-(a,b) and \eqref{eq:rep-xb} of Theorem \ref{res-02}
is called a \emph{dominating point}.

\begin{definition}[Dominating point]\label{defdomp}
Let $C\subset\XXo$ be a convex set such that $C\cap\XL$ is
$\sigma(\XL,\YL)$-closed. The point $\xh\in\Xx$ is called a
$\Gamma^*$-\emph{dominating point} if
\begin{itemize}
    \item[(a)]
    $\xh\in C\cap\dom\Gamma^*$
    \item[(b)]
    there exists some linear form $\ot$ on $\Xx$ such that $\langle
    \ot,\xh\rangle\leq \langle \ot,x\rangle$ for all $x\in C\cap\dom\Gamma^*,$
     \item[(c)]
    $\langle\ot,\theta(\cdot)\rangle$ is measurable and
    $\xh=\IZ\theta(z)\gamma'(z,\langle\ot,\theta(z)\rangle)\,R(dz).$
\end{itemize}
    \index{domp@dominating point, see Definition \ref{defdomp}}
\end{definition}

In the special case where $\Gamma^*$ is replaced by the Cram\'er
transform $\Xi$ defined at \eqref{cramer}, taking Lemma
\ref{astuce} into account, this definition becomes the following
one.

\begin{definition}\label{dompointNey}
Let $C\subset\XXo$ be a convex set such that $C\cap\XL$ is
$\sigma(\XL,\YL)$-closed. The point $\xh\in\Xx$ is called a
$\Xi$-\emph{dominating point} if
\begin{itemize}
    \item[(a)]
    $\xh\in C\cap\dom\Xi$
    \item[(b)]
    there exists some linear form $\ot$ on $\Xx$ such that $\langle
    \ot,\xh\rangle\leq \langle \ot,x\rangle,$ for all $x\in C\cap\dom\Xi$ and
    \item[(c)]
    $\langle\ot,\theta(\cdot)\rangle$ is measurable and
    $\displaystyle{\xh=\IZ\theta(z)\frac{\exp(\langle\ot,\theta(z)\rangle)}{Z(\ot)}\,R(dz)}$
    where $Z(\ot)$ is the unit mass normalizing constant.
    \index{Xidomp@$\Xi$-dominating point, see Definition \ref{dompointNey}}
\end{itemize}
\end{definition}

Note that this definition is slightly different from the ones
proposed by Ney \cite{Ney84} and Einmahl and Kuelbs \cite{EK96}
since $C$ is neither supposed to be an open set nor to have a
non-empty interior and $\xh$ is not assumed to be a boundary point
of $C.$ The above integral representation (c) is \eqref{eq-31}.

We are going to investigate some relations between dominating
points and entropic projections. In the case where the constraint
is good, Theorem \ref{res-02} and Proposition \ref{res:Iproj}
state that the generalized entropic projection is the entropic
projection $Q_*=\Qh,$ the minimizer $\xb$ is the dominating point
of $C,$ it is related to $Q_*$ by the identity:
\begin{equation}\label{eq-32}
    \xb=\langle\theta,Q_*\rangle.
\end{equation}
We now look at the situation where the constraint is critical. As
remarked at Example \ref{exCsi1}, the above equality may fail. A
necessary and sufficient condition (in terms of the function
$\Gamma^*$) for $\xb$ to satisfy \eqref{eq-32} is obtained at
Theorem \ref{res:dompoint}.

Recall that the extended entropy $\Ib$ is given by \eqref{III}:
$\Ib(\ell)=I(\ell^a)+I^s(\ell^s),$ $\ell=\ell^a+\ell^s\in\Ll'$ and
define for all $x\in\Xx,$
\begin{eqnarray*}
  \Jb(x) &:=& \inf\{\Ib(\ell); \ell\in\Ll',\langle\theta,\ell\rangle=x\} \\
  J(x) &:=& \inf\{I(\ell); \ell\in\LlsR,\langle\theta,\ell\rangle=x\}\\
  J^s(x) &:=& \inf\{I^s(\ell); \ell\in\Ll^s,\langle\theta,\ell\rangle=x\}
  \index{JJJ@$\Jb,J^s$}
  \index{J@$J$}
\end{eqnarray*}
Because of the decomposition $\Ll'\simeq\LlsR\oplus\Ll^s,$ we
obtain for all $x\in\Xx,$
\begin{eqnarray*}
  \Jb(x) &=& \inf\{I(\ell_1)+I^s(\ell_2);\ell_1\in\LlsR, \ell_2\in\Ll^s, \langle\theta,\ell_1+\ell_2\rangle=x\} \\
   &=& \inf\{J(x_1)+J^s(x_2); x_1,x_2\in\Xx, x_1+x_2=x\} \\
   &=& J\Box J^s(x)
\end{eqnarray*}
where $J\Box J^s$ is the inf-convolution of $J$ and $J^s.$ By
Theorem \ref{res-03}-a, $\Jb=\Gamma^*$ and if $\Jb(x)<\infty,$
there exists $\ell_x\in\Ll'$ such that
$\langle\theta,\ell_x\rangle=x$ and $\Jb(x)=\Ib(\ell_x).$ Let us
define
\[
x^a:= \langle\theta,\ell_x^a\rangle \quad\textrm{ and }\quad
x^s:=\langle\theta,\ell_x^s\rangle.
    \index{xxxas@$x^a, x^s$}
\]
These definitions make sense since $\ell_x^a$ is the unique
(common) absolutely continuous part of the minimizers of $\Ib$ on
the closed convex set $\{\ell\in\Ll';
\langle\theta,\ell\rangle=x\},$ see Theorem \ref{res-07}-a. Of
course, we have $$x=x^a+x^s$$ and as
$\Jb(x)=\Ib(\ell_x)=I(\ell_x^a)+I^s(\ell_x^s)\geq
J(x^a)+J^s(x^s)\geq J\Box J^s(x)= \Jb(x),$ we get the following
result.
\begin{proposition}\label{prop51}
For all $x\in\dom\Jb,$ we have:
\begin{center} \(
  \Jb(x) = J(x^a)+J^s(x^s),
  J(x^a) = I(\ell_x^a)\) and
  \(J^s(x^s) = I^s(\ell_x^s).\)
  \end{center}
\end{proposition}

Now, let us have a look at the dual equalities. The
\emph{recession function} of $\Gamma^*$ is defined for all $x$ by
\begin{equation}\label{eq-60}
    \widetilde{\Gamma^*}(x):=\lim_{t\rightarrow
+\infty}\Gamma^*(tx)/t\in (-\infty,+\infty].
    \index{Gammastarti@$\widetilde{\Gamma^*},$ see \eqref{eq-60}}
\end{equation}
\begin{definition}[Recessive $x$]\label{def-01}
 Let us say that
$x$ is \emph{recessive} for $\Gamma^*$ if for some $\delta>0$ and
$\xi\in\Xx,$
$\Gamma^*(x+t\xi)-\Gamma^*(x)=t\widetilde{\Gamma^*}(\xi)$ for all
$t\in (-\delta,+\infty).$ It is said to be \emph{non-recessive}
otherwise.
    \index{recx@recessive $x,$ see Definition \ref{def-01}}
    \index{nonrecx@non-recessive $x,$ see Definition \ref{def-01}}
\end{definition}
\begin{proposition}
We have $\Jb=\Gamma^*$ and $J^s=\widetilde{\Gamma^*}.$ Moreover,
$J(x)=\Gamma^*(x)$ for all non-recessive $x\in\Xx.$
\end{proposition}
\proof We have already noted that $\Jb=\Gamma^*$ and by
\cite[Theorem 2.3]{Leo01c}, we get: $J^s=\iota_{\dom\Gamma}^*:$
the support function of $\dom\Gamma.$ Therefore, it is also the
recession function of $\Gamma^*.$ Hence, we have
$\Gamma^*=\Jb=J\Box J^s=J\Box \widetilde{\Gamma^*}.$
\\
Comparing $\Gamma^*= J\Box\widetilde{\Gamma^*}$ with the general
identity $\Gamma^*= \Gamma^*\Box\widetilde{\Gamma^*},$ we obtain
that $J(x)=\Gamma^*(x),$  for all non-recessive $x\in\Xx.$
\endproof

\begin{proposition}\label{res-08}
For all $x\in\dom\Gamma^*,$ we have:
\[
\Gamma^*(x)=\Gamma^*(x^a)+\widetilde{\Gamma^*}(x^s).
\]
Moreover, $x$ is non-recessive if and only if $x^s=0.$ In
particular, $x^a$ is non-recessive.
\end{proposition}
\proof By \eqref{III}, we have
$\Ib(\ell_x)=I(\ell_x^a)+\widetilde{I}(\ell_x^s)$ where
$\widetilde{I}$ is the recession function of $I.$ It follows that
$\Jb(x)=J(x_a)+J^s(x^s),$ since $J(x_a)=I(\ell_x^a)$ (Proposition
\ref{prop51}) and the recession function
 of $\Jb$ is $J^s.$ To show this, note that (see \cite{RW98})
\begin{itemize}
    \item[-]
    $I^s$ is the recession function of $\Ib,$
    \item[-]
    the epigraph of $x\mapsto\inf\{f(\ell);\ell, T\ell=x\}$ (with $T$ a linear
    operator) is ``essentially" a linear projection of the epigraph of
    $f,$ (let us call it an inf-projection)
    \item[-]
    the epigraph of the recession function is the recession cone
    of the epigraph and
    \item[-]
    the inf-projection of a recession cone is the recession cone of
    the inf-projection.
\end{itemize}
 The first result now follows from $\Jb=\Gamma^*.$ The same set
of arguments also yields the second statement.
\endproof

\begin{theorem}\label{res:dompoint}
Let us assume that the hypotheses of Theorem \ref{res-03} hold and
that $C\cap\icordom\Gamma^*\not=\emptyset.$
\begin{itemize}
    \item[(a)]
     Then, a minimizer $\xb$ of $\Gamma^*$ on the set $C$ is a $\Gamma^*$-dominating
    point of  $C$ if and only if $\xb$ is non-recessive. This is also
    equivalent to the following statement: ``all the solutions of the minimization problem
   \eqref{pbc} are absolutely continuous with respect to $R.$" In such a case the
    solution of \eqref{pbc} is unique and it matches the solution of
   \eqref{pc}.
    \item[(b)]
    In particular when $\Gamma^*$ admits a degenerate recession
function, i.e.\! $\widetilde{\Gamma^*}(x)=+\infty$ for all
$x\not=0,$ then the minimizer $\xb$ is a $\Gamma^*$-dominating
point of $C.$
    \item[(c)] The same statements hold with $\Gamma^*$ replaced by $\Xi.$
\end{itemize}
\end{theorem}
\proof This is a direct consequence of Theorem \ref{res-03},
Proposition \ref{res-08} and Lemma \ref{astuce}.
\endproof

\begin{remark}\label{rem-steep} \textit{A remark about the steepness of the log-Laplace
transform.} In \cite[Thm 1]{Kue00}, with the setting of Section
\ref{sec:entropy2} where $\Xx$ is a Banach space, Kuelbs proves a
result that is slightly different from statement (b) of the above
theorem. It is proved that the existence of a $\Xi$-dominating
point for \emph{all} convex sets $C$ with a nonempty topological
interior is equivalent to some property of the G\^ateaux
derivative of the log-Laplace transform
$y\in\Xx'\mapsto\Lambda(y)=\log\int_{\Xx} \exp(\langle
y,x\rangle)\, R\circ\theta^{-1}(dx)$ on the boundary of its
domain. This property is an infinite dimensional analogue of the
steepness of the log-Laplace transform. It turns out that it is
equivalent to the following assumption: the Cram\'er transform
$\Xi=\Lambda^*$ admits a degenerate recession function.
\end{remark}

\begin{example}\label{exCsi2} \textit{Csisz\'ar's example
continued.} Recall that by Lemma \ref{astuce},
$\Xi(x)=\Lambda^*(x)=\Gamma^*(1,x)$ so that $\Gamma$ and $\Lambda$
play the same role. Clearly, $\dom\Lambda=(-\infty,1]$ and
$\Lambda'(1^-)=\int_{[0,\infty)}z\,P_1(dz):=x_*<\infty.$  It
follows that for all $x\ge x_*,$ $\Xi(x)-\Xi(x_*)=x-x_*.$

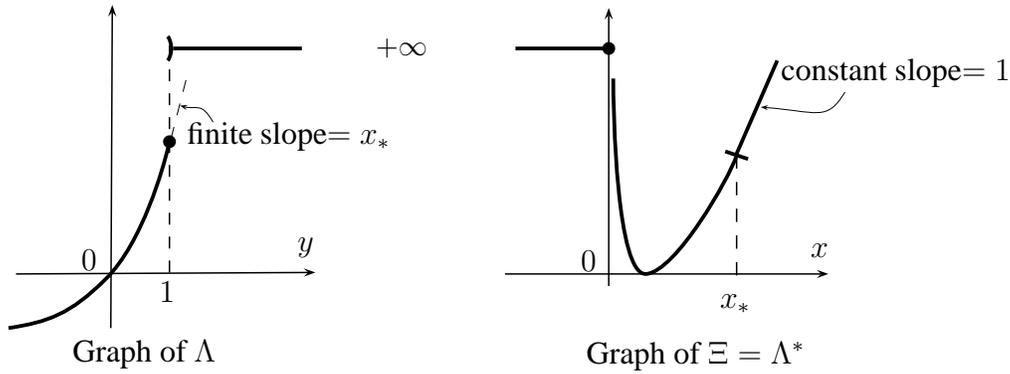
\begin{figure}[h]
\begin{center}
\scalebox{1} 
{
\begin{pspicture}(0,-2.4684374)(13.221875,2.4384375)
\psline[linewidth=0.02cm,arrowsize=0.05291667cm
3.01,arrowlength=1.4,arrowinset=0.4]{<-}(1.38,2.4284375)(1.36,-1.8715625)
\psline[linewidth=0.02cm,arrowsize=0.05291667cm
3.0,arrowlength=1.4,arrowinset=0.4]{->}(0.1,-1.1515625)(4.08,-1.1315625)
\psline[linewidth=0.02cm,linestyle=dashed,dash=0.16cm
0.16cm](2.14,1.8884375)(2.14,-1.1315625)
\psbezier[linewidth=0.05,dotsize=0.07055555cm
2.0]{-*}(0.0,-1.8715625)(0.52,-1.7915626)(0.82,-1.6915625)(1.3,-1.2115625)(1.78,-0.7315625)(2.06,0.2684375)(2.14,0.6084375)
\psline[linewidth=0.012cm,linestyle=dashed,dash=0.16cm
0.16cm](2.16,0.6484375)(2.36,1.4284375)
\psline[linewidth=0.05cm,tbarsize=0.07055555cm
5.0,rbracketlength=0.15]{)-}(2.14,1.8484375)(3.9,1.8484375)
\psline[linewidth=0.02cm,arrowsize=0.05291667cm
3.0,arrowlength=1.4,arrowinset=0.4]{->}(6.6,-1.1515625)(10.9,-1.1515625)
\psline[linewidth=0.02cm,arrowsize=0.05291667cm
3.0,arrowlength=1.4,arrowinset=0.4]{->}(7.98,-1.3115625)(7.98,2.3884375)
\psbezier[linewidth=0.05,tbarsize=0.07055555cm
5.0]{-|*}(8.04,1.4484375)(8.04,1.0284375)(8.08,-1.1915625)(8.48,-1.1515625)(8.88,-1.1115625)(9.58,0.1284375)(9.68,0.4284375)
\psline[linewidth=0.05cm](9.68,0.4484375)(10.22,1.6884375)
\psline[linewidth=0.05cm,dotsize=0.07055555cm
2.0]{*-}(7.98,1.8484375)(6.74,1.8484375)
\usefont{T1}{ptm}{m}{n}
\rput(5.2275,1.8484375){\small $+\infty$}
\usefont{T1}{ptm}{m}{n}
\rput(1.0714062,-0.9615625){$0$}
\usefont{T1}{ptm}{m}{n}
\rput(3.9514062,-0.8015625){$y$}
\usefont{T1}{ptm}{m}{n}
\rput(7.711406,-0.9815625){$0$}
\psline[linewidth=0.02cm,linestyle=dashed,dash=0.16cm
0.16cm](9.68,0.4284375)(9.68,-1.1715626)
\usefont{T1}{ptm}{m}{n}
\rput(9.6614065,-1.5415626){$x_*$}
\psbezier[linewidth=0.01,arrowsize=0.05291667cm
2.0,arrowlength=1.4,arrowinset=0.4]{<-}(10.0,1.0684375)(10.76,1.0084375)(10.88,0.8684375)(11.12,1.2684375)
\usefont{T1}{ptm}{m}{n}
\rput(11.794844,1.5184375){constant slope$=1$}
\usefont{T1}{ptm}{m}{n}
\rput(2.1114063,-1.3815625){$1$}
\psbezier[linewidth=0.01,arrowsize=0.05291667cm
2.0,arrowlength=1.4,arrowinset=0.4]{<-}(2.2784615,1.0884376)(2.5,1.0284375)(2.56,1.1484375)(2.72,0.8684375)
\usefont{T1}{ptm}{m}{n}
\rput(3.7240624,0.6784375){finite slope$=x_*$}
\usefont{T1}{ptm}{m}{n}
\rput(10.781406,-0.8415625){$x$}
\usefont{T1}{ptm}{m}{n}
\rput(1.8045312,-2.2215624){Graph of $\Lambda$}
\usefont{T1}{ptm}{m}{n}
\rput(9.134531,-2.2415626){Graph of $\Xi=\Lambda^*$}
\end{pspicture}
}
\end{center}
 \caption{The case: $R(dz)= \frac 1{a_0}\frac{e^{-z}}{1+z^3}\,dz$}\label{fig3}
 \end{figure}

We deduce from this that $(x_*,\infty)$ is a set of recessive
points. By Theorem \ref{res:dompoint}, they cannot be dominating
points. Note also that the log-Laplace transform $\Lambda$ is not
steep.
\end{example}

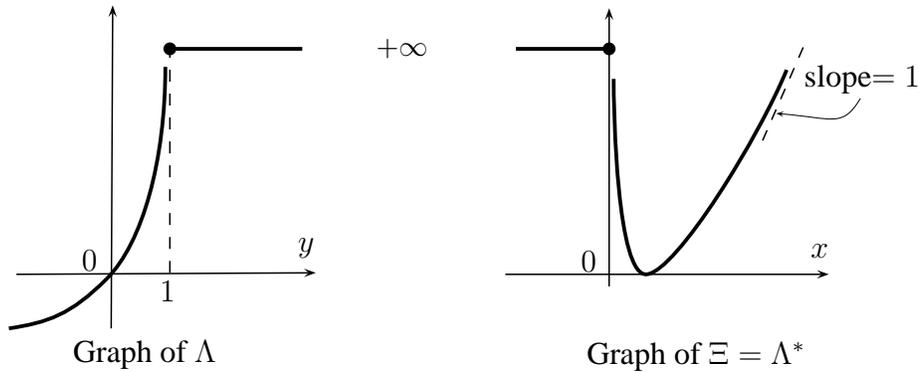
\begin{figure}[ht]
\begin{center}
\scalebox{1} 
{
\begin{pspicture}(0,-2.4684374)(12.121875,2.4384375)
\psline[linewidth=0.02cm,arrowsize=0.05291667cm
3.01,arrowlength=1.4,arrowinset=0.4]{<-}(1.38,2.4284375)(1.36,-1.8715625)
\psline[linewidth=0.02cm,arrowsize=0.05291667cm
3.0,arrowlength=1.4,arrowinset=0.4]{->}(0.1,-1.1515625)(4.08,-1.1315625)
\psline[linewidth=0.02cm,linestyle=dashed,dash=0.16cm
0.16cm](2.14,1.8884375)(2.14,-1.1315625)
\psbezier[linewidth=0.05](0.0,-1.8715625)(0.52,-1.7915626)(0.82,-1.6915625)(1.3,-1.2115625)(1.78,-0.7315625)(2.08,0.3684375)(2.08,1.6084375)
\psline[linewidth=0.05cm,dotsize=0.07055555cm
2.0]{*-}(2.14,1.8484375)(3.9,1.8484375)
\psline[linewidth=0.02cm,arrowsize=0.05291667cm
3.0,arrowlength=1.4,arrowinset=0.4]{->}(6.6,-1.1515625)(10.9,-1.1515625)
\psline[linewidth=0.02cm,arrowsize=0.05291667cm
3.0,arrowlength=1.4,arrowinset=0.4]{->}(7.98,-1.3115625)(7.98,2.3884375)
\psbezier[linewidth=0.05](8.04,1.4484375)(8.04,1.0284375)(8.08,-1.1915625)(8.48,-1.1515625)(8.88,-1.1115625)(9.98,0.6884375)(10.34,1.5684375)
\psline[linewidth=0.02cm,linestyle=dashed,dash=0.16cm
0.16cm](10.02,0.6284375)(10.56,1.8684375)
\psline[linewidth=0.05cm,dotsize=0.07055555cm
2.0]{*-}(7.98,1.8484375)(6.74,1.8484375)
\usefont{T1}{ptm}{m}{n}
\rput(5.2275,1.8484375){\small $+\infty$}
\usefont{T1}{ptm}{m}{n}
\rput(1.0714062,-0.9615625){$0$}
\usefont{T1}{ptm}{m}{n}
\rput(3.9514062,-0.8015625){$y$}
\usefont{T1}{ptm}{m}{n}
\rput(7.711406,-0.9815625){$0$}
\psbezier[linewidth=0.01,arrowsize=0.05291667cm
2.0,arrowlength=1.4,arrowinset=0.4]{<-}(10.2,0.9884375)(10.96,0.9284375)(11.08,0.7884375)(11.32,1.1884375)
\usefont{T1}{ptm}{m}{n}
\rput(11.342188,1.4384375){slope$=1$}
\usefont{T1}{ptm}{m}{n}
\rput(2.1114063,-1.3815625){$1$}
\usefont{T1}{ptm}{m}{n}
\rput(10.781406,-0.8415625){$x$}
\usefont{T1}{ptm}{m}{n}
\rput(1.8045312,-2.2215624){Graph of $\Lambda$}
\usefont{T1}{ptm}{m}{n}
\rput(9.134531,-2.2415626){Graph of $\Xi=\Lambda^*$}
\end{pspicture}
}
\end{center}
 \caption{The case: $R(dz)= e^{-z}\,dz$}\label{fig4}
 \end{figure}

\begin{remark}\textit{A remark about Csisz\'ar's example.} For
comparison, suppose that $R$ in Csisz\'ar's Example \ref{exCsi1}
and \ref{exCsi2} is replaced by its slight modification $R(dz)=
e^{-z}\,dz:$ the exponential law with parameter 1. Then
$\Lambda(y)=-\log(1-y)+\iota_{\{y\ge1\}}$ and
$\Xi(x)=\Lambda^*(x)=x-1-\log x+\iota_{\{x\le0\}}.$ Their graphic
representations are given at Figure \ref{fig4} and one sees that
for any $c>0,$ $R$ admits the entropic projection
$P_y(dz)=(1-y)e^{(y-1)z}\,dz$ which is the exponential law with
parameter $(1-y)$ with $y$ such that
$1/(1-y)=\int_{[0,\infty)}z\,P_y(dz)=c.$ In this case, for any
$c>0$ the generalized entropic projection is the entropic
projection, $\Lambda$ is steep and $\Xi$ has no recessive points.
\end{remark}

\section{Conditional laws of large numbers}\label{sec:condLLN}

In this last section, we give a probabilistic interpretation of
the singular component of the generalized $I$-projection in terms
of a conditional LLN.

\subsection{Conditional laws of large numbers}

Conditional laws of large numbers are already well-known, see
\cite{Az80} or \cite{SZ91} for instance. Theorem \ref{res-09}
below is a general statement of a conditional LLN which is
inspired from \cite{SZ91} and \cite[Section 7.3]{DZ}. Its proof is
given at the Appendix \ref{sec:proofcondLLN}.

Let $\{L_n\}\index{Ln@$L_n$}$ be a sequence of random vectors in
the algebraic dual space $\LLo\index{LLo@$\LLo$}$ of some vector
space $\UUo\index{UUo@$\UUo$}.$ As a typical instance, one can
think of random measures $L_n$ with $\UUo$ a function space. Let
$T:\LLo\rightarrow\XXo$ be a linear operator with values in
another vector space $\XXo.$ We are going to investigate the
behavior of the conditional law $\mathbb{P}(L_n\in\cdot\mid
TL_n\in C)$ of $L_n$ as $n$ tends to infinity, for some measurable
set $C$ in $\XXo.$
 It appears that this type of
conditional law of large numbers is connected with large
deviations. We assume that $\{L_n\}$ obeys the LDP with a good
rate function $I$ in $\LLo$ endowed with the weak topology
$\sigma(\LLo,\UUo)$ and the associated cylinder $\sigma$-field. It
is also clear that one should assume that $\mathbb{P}(TL_n\in
C)>0$ for all $n,$ not to divide by zero. To overcome this
restriction, we look at $\mathbb{P}(L_n\in\cdot\mid TL_n\in
C_\delta)$ where $C_\delta$ tends to $C$ as $\delta$ tends to
zero.
\\
Let us assume that $\XXo$ is a topological vector space with its
Borel $\sigma$-field and that $T:\LLo\rightarrow\XXo$ is
continuous. The contraction principle tells us that
\begin{displaymath}
  X_n:= TL_n
\end{displaymath}
obeys the LDP in $\XXo$ with the rate function
\begin{displaymath}
J(x)=\inf\{I(\ell) ; \ell\in\LLo, T\ell=x\}.
    \index{J@$J$}
\end{displaymath}
The following set of hypotheses is based on a framework which
appears in \cite{SZ91} and \cite[Section 7.3]{DZ}.

\par\medskip\noindent\textbf{Assumptions (B).}\index{Assumptions (B)}\
\begin{itemize}
  \item[(B$_L$)]\index{Assumptions (B)!B1@(B$_L$) on $\{L_n\}$}
The sequence $\{L_n\}$ obeys the LDP in $\LLo$ with a
  \emph{good} rate function $I.$ This means that $I$ is inf-compact.

  \item[(B$_T$)]\index{Assumptions (B)!B2@(B$_T$) on $T$}
   The linear operator $T:\LLo\rightarrow \XXo$ is continuous.
    \item[(B$_C$)]\index{Assumptions (B)!B3@(B$_C$) on the constraint set $C$} \textit{Assumptions on $C$}.\
As a convention, we write $J(C)$ for $\inf_{x\in C} J(x).$
\begin{enumerate}[(1)]
 \item
      $J(C)<\infty.$
    \item
     The set $C$ is closed. It is the limit as $\delta$ decreases
      to 0: $C= \cap_{\delta}C_\delta,$ of
      a decreasing family of closed sets $C_{\delta}$ in $\XXo$ such
      that for all $\delta>0$ and all $n\geq 1,$ $\mathbb{P}(X_n\in C_\delta)>0$
 \item
       and one of the following statements
     \begin{enumerate}[(a)]
       \item
          $C_{\delta}=C$ for all $\delta >0$ and $J(\mathrm{int\
          }C)=J(C),$ or
       \item
          $C\subset \mathrm{int \ }C_{\delta}$ for all $\delta >0.$
     \end{enumerate}
       is fulfilled.
\end{enumerate}
\end{itemize}

Let $\mathcal{G}\index{G@$\mathcal{G}$}$ be the set of all
solutions of the following minimization problem:
\begin{equation}
\label{eq:cond4}
  \textsl{minimize } I(\ell) \textsl{ subject to } T\ell\in C,\quad\ell\in\LLo.
\end{equation}
Similarly, let $\mathcal{H}\index{H@$\mathcal{H}$}$ be the set of
all solutions of the following minimization problem:
\begin{equation}
  \label{eq:cond3}
  \textsl{minimize } J(x) \textsl{ subject to } x\in C,\quad x\in\XXo.
\end{equation}
We can now state a result about conditional laws of large numbers
which is proved at the Appendix.

\begin{theorem}\label{res-09} Assume \emph{(B)}.
For all open subset $G$ of $\LLo$ such that $\mathcal{G}\subset G$
and all open subset $H$ of $\XXo$ such that $\mathcal{H}\subset
H,$ we have
\begin{eqnarray*}
&&\limsupdn \frac 1n \log\mathbb{P}(L_n\not\in G\mid TL_n\in
C_\delta)<0
\textsl{\ and}\\
&&\limsupdn \frac 1n \log\mathbb{P}(X_n\not\in H\mid X_n\in
C_\delta)<0.
\end{eqnarray*}
In particular, if $C$ is convex and the rate functions $I$ and $J$
are strictly convex, we have the conditional laws of large
numbers:
\begin{eqnarray*}
  \lim_\delta \lim_n\mathbb{P}(L_n\in\cdot\mid TL_n\in C_\delta)&=&\delta_{\lb}\\
  \lim_\delta \lim_n\mathbb{P}(X_n\in\cdot\mid X_n\in C_\delta)&=&\delta_{\xb}
\end{eqnarray*}
where the limits are understood with respect to the usual weak
topologies of probability measures, $\lb$ is the unique solution
to the convex minimization problem \eqref{eq:cond4} and $\xb=T\lb$
is the unique solution to \eqref{eq:cond3}.
\end{theorem}

\subsection{Empirical measures of independent samples}

Let $Z_1,Z_2,\dots$ be an iid sequence of $\ZZ$-valued random
variables which are $R$-distributed. The LDP for $ L_n=\moyn
\delta_{Z_i}$ is described by Sanov's theorem: $L_n$ obeys the LDP
in $\PZ$ endowed with the topology $\sigma(\PZ,B_\ZZ)$  with the
relative entropy $I(.|R)$ as its rate function, see \cite[Thm
6.2.10]{DZ}. Here, $B_\ZZ\index{BZ@$B_\ZZ,$ bounded measurable
functions on $\ZZ$}$ denotes the space of all bounded measurable
functions on $\ZZ.$

One can improve this result: this LDP still holds in
$P_{\mathrm{exp}}\subset\PZ$ with the topology
$\sigma(P_{\mathrm{exp}},\Eexp)$ where
 $\Eexp=\{u;\forall a>0,\IZ e^{a|u|}\,dR<\infty\}
 \index{Eexp@$\Eexp$}$
is $E_\tau(R)$ without identifying the $R$-a.e.\! equal functions
and $P_{\mathrm{exp}}=\{P\in\PZ; \IZ |u|\,dP<\infty,\forall u\in
\Eexp\},\index{Pexp@$P_{\mathrm{exp}}$}$ see \cite{Sch98,LeoN02}.

A further improvement is obtained in \cite{LeoN02}: $L_n$ obeys
the LDP in $\OZ$ endowed with the topology $\sigma(\OZ,\Lexp)$
with the extended relative entropy $\Ib(.|R)$ as its rate
function. Here, $\Lexp=\{u;\exists a>0,\IZ e^{a|u|}\,dR<\infty\}
\index{Lexp@$\Lexp$}$ is $L_\tau(R)$ without identifying the
$R$-a.e.\! equal functions.
\\
Hence, under the assumptions (B$_C$) on $\{C_\delta\}$ and $C,$ we
can apply Theorem \ref{res-09} with $I=\Ib(.|R)$ and $J=\Xi$ which
are defined at \eqref{ext-entrop} and \eqref{cramer}. More
precisely, if the hypotheses of Proposition \ref{res:entrop1} are
fulfilled, we see that
\begin{equation}\label{eq-41}
  \lim_\delta \lim_n\mathbb{P}\left(\moyn\delta_{Z_i}\in\cdot\mid
  \moyn \theta(Z_i)\in C_\delta\right)=\delta_{\widehat{P}}
\end{equation}
where $\widehat{P}$ is characterized by \eqref{eq-33} and
\eqref{eq-35} and this limit corresponds to the usual weak
topology on the set of probability measures on $P_{\mathrm{exp}}$
furnished with the topology $\sigma(P_{\mathrm{exp}},\Eexp).$ If
the hypotheses of Proposition \ref{res:entrop2} are fulfilled and
one of the following assumptions
\begin{enumerate}
    \item
          $\XXo=\R^K$ and  $\CC\cap\icordom
          I\not=\emptyset$ or
    \item
         $\|\cdot\|_{\tau^*}$-$\inter(\CC)\not=\emptyset.$
\end{enumerate}
holds, Theorem \ref{res-09} and Proposition \ref{res:entrop2} tell
us that \eqref{eq-41} still holds, more precisely
\begin{equation*}
  \lim_\delta \lim_n\mathbb{P}\left(\moyn\delta_{Z_i}\in\cdot\mid \moyn \theta(Z_i)\in C_\delta\right)=\delta_{P_\diamond}
\end{equation*}
with respect to the same topology
$\sigma(P_{\mathrm{exp}},\Eexp),$ where $P_\diamond$ is
characterized at \eqref{eq-34}. See \cite{Csi84} for similar
results w.r.t.\! $\sigma(\PZ,B_\ZZ).$

\subsection{A probabilistic interpretation of the singular
component}\label{sec-interp}

On the other hand, considering $L_n\in\PZ\subset\OZ,$
\eqref{eq-41} fails in $\OZ$ endowed with topology
$\sigma(\OZ,\Lexp).$ Instead of this, the laws
$\mathbb{P}\left(L_n\in\cdot\mid \moyn \theta(Z_i)\in
C_\delta\right)$ admit cluster points which are probability
measures on ${\OZ}$ whose support consists of solutions $\lh$ to
the minimization problem \eqref{eq-22}. We know that
$\lh^a=P_\diamond$ for any $\lh.$

By Proposition \ref{res-B4}, $\langle\lh^s,u\rangle=0$ for all
$u\in\Eexp.$ In particular, the mass of $\lh^s$ is
$\langle\lh^s,\mathbf{1}\rangle=0.$ Nevertheless, $\lh^s$ may be
nonzero. This reflects the fact for some indices $1\le i \le k_n$
(we choose the first ones without loss of generality) with a
vanishing ratio: $\limn k_n/n=0,$ and for some functions $u$ in
$\Lexp\setminus\Eexp,$ $\frac 1n\sum_{i=1}^{k_n}u(Z_i)$ does not
vanish as $n$ tends to infinity. Meanwhile, the remaining
variables are such that $\limn\frac 1{n-k_n}\sum_{i=
k_n+1}^n\delta_{Z_i}=P_\diamond$ with respect to
$\sigma(\OZ,\Lexp).$

\begin{example}\label{exCsi3}
Consider Example \ref{exCsi1} with $c>x_*.$ Suppose that
$L_n-\frac 1n\delta_{Z_1}$ is close to $P_\diamond=P_1$ and
$Z_1/n\ge c-x_*$ so that the event $\moyn Z_i\ge c-1/n$ is
realized. When $n$ is large, this happens  with a probability
$p_n\simeq\exp[-(n-1)I(P_\diamond|R)]\mathbb{P}(Z_1\ge n(c-x_*))$
and with the notation of Example \ref{exCsi2}, $\limn -\frac
1n\log
p_n=I(P_\diamond|R)+(c-x_*)=\Gamma^*(x_*)+[\Gamma^*(c)-\Gamma^*(x_*)]=\Gamma^*(c).$
Therefore, $p_n$ has the optimal logarithmic behavior.
\\
What is performed by the first particle  $i=1,$ may also be
performed by any collection of $k_n$ different particles with a
probability $p_n(a_1^n,\dots,a_{k_n}^n)\simeq
\begin{pmatrix}n\\k_n\end{pmatrix}
 \exp[-(n-k_n)I(P_\diamond|R)]\mathbb{P}(Z_1\ge a_1^n)\cdots\mathbb{P}(Z_{k_n}\ge a_{k_n}^n)$
 for any $a_1^n,\dots,a_{k_n}^n$ such that
\begin{equation*}
    \left\{
    \begin{array}{lcl}
      \limn k_n/n&=&0, \\
       \limn a_k^n&=&+\infty,\quad \forall 1\le k\le k_n,\\
       \limn(a_1^n+\cdots+a_{k_n}^n)/n&=&c-x_*.\\
    \end{array}
    \right.
\end{equation*}
Again, $p_n(a_1^n,\dots,a_{k_n}^n)$ is logarithmically optimal
since $\limn -\frac 1n\log
p_n(a_1^n,\dots,a_{k_n}^n)=I(P_\diamond|R)+\limn
(a_1^n+\cdots+a_{k_n}^n)/n=I(P_\diamond|R)+(c-x_*)=\Gamma^*(c).$

A \emph{vanishing ratio} of particles performing very large jumps
is responsible for the appearance of the singular component.
\end{example}

For other examples of minimization problems \eqref{eq-22} whose
solutions exhibit a nonzero singular component, one can have a
look at \cite[Section 7.4]{Leo08}.

\subsection{Empirical measures with random weights}
In \cite{DDCG90,GG97}, one gives a Bayesian interpretation of the
conditional LLN of the empirical measures with random weights
$L_n=\moyn W_i\delta_{z_i},$ see \eqref{eq-27}. Let the weights
$W$ be iid  random variables with a common Cram\'er transform
$\gamma^*.$  Let $\ZZ$ be a topological space with its Borel
$\sigma$-field. Denote $C_\ZZ\index{CZ@$C_\ZZ,$ bounded continuous
functions on $\ZZ$}$ the space of all bounded continuous functions
on $\ZZ.$ The LDP for $L_n$ is given at \cite[Thm 7.2.3]{DZ}: if
$\gamma^*$ satisfies \eqref{A4}, $L_n$ obeys the LDP in $\MZ$ with
the topology $\sigma(\MZ,C_\ZZ)$ and the rate function
$I=I_{\gamma^*}.$

 Under the hypotheses of Theorems \ref{res-05} and
\ref{res-09}, we obtain
\begin{equation*}
  \lim_\delta \lim_n\mathbb{P}\left(\moyn W_i\delta_{z_i}\in\cdot
  \mid \moyn W_i\theta(z_i)\in C_\delta\right)=\delta_{\Qd}
\end{equation*}
where $\Qd$ is described at \eqref{eq-18} and the topology on
$\MZ$ is  $\sigma(\MZ,C_\ZZ).$

Under the more restrictive hypotheses of Proposition
\ref{res:Iproj}, we have $\Qd=\Qh$ which is the solution of
\eqref{pc}. If the weights $W$ are nonnegative and
$\mathbb{E}W=1,$ the reference probability measure $R$ is
interpreted as the prior distribution while $\Qh$ is a posterior
distribution which is selected in $\CC$ by means of \eqref{pc}.
Hence, it is possible to estimate $\Qh$ with $L_n.$ It is the aim
of the method of Maximum Entropy in the Mean (MEM) to provide an
effective simulation procedure of the posterior distribution.

The reverse entropy corresponds to a $\gamma^*$ which doesn't
satisfy \eqref{A4}; we have $\Ll=L_\infty$ and $\Ll'=L_1R\oplus
L_\infty'.$ This situation is investigated in details in
\cite{GG97} with a compact space $\ZZ.$ This compactness allows to
restrict the attention to the subspace $\MZ\subset L_1R\oplus
L_\infty'$ which is decomposed into the direct sum of the spaces
of measures which are absolutely continuous and singular with
respect to $R.$ It is proved that the $(\delta,n)$-sequence
$\mathbb{P}\left(\moyn W_i\delta_{z_i}\in\cdot\mid \moyn
W_i\theta(z_i)\in C_\delta\right)$ admits cluster points in $\MZ$
whose support consist of solutions of \eqref{pbc} which may have a
singular component. The probabilistic interpretation of the Dirac
part of this singular component is: a \emph{non-vanishing ratio}
of indices $i$ accumulate in areas where the $z_i$ are close to
each other, with large values of $W_i.$

One can expect that $L_n$ obeys LDPs  with respect to topologies
weakened by unbounded functions, but no results of this kind
appear in the literature. If this holds true, for any $\gamma^*$
satisfying \eqref{A4} but such that $\lmax$ isn't
$\Delta_2$-regular, an interpretation of the singular component in
terms a vanishing ratio of indices is still available, as in
Section \ref{sec-interp}.

\appendix

\section{Complements for the statements of Theorems \ref{res-02} and \ref{res-03}}\label{annex}

 The decomposition into positive and negative parts of
linear forms is necessary to state the extended dual problem which
is needed for the characterization of the minimizers. If $\lambda$
is not an even function, one has to consider
\begin{equation*}
    \left\{
    \begin{array}{l}
       \lambda_+(z,s)=\lambda(z,|s|) \\
       \lambda_-(z,s)=\lambda(z,-|s|) \\
    \end{array}
    \right.
\end{equation*}
which are Young functions and the corresponding Orlicz spaces
$L_{\lambda_\pm}.$

\subsection*{Some definitions needed for stating Theorem \ref{res-02}}\

\par\smallskip\noindent \textsf{The cone $K_\lambda$.}\index{Klambda@$K_\lambda$}\
It is the cone of all measurable
 functions $u$ with a positive part $u_+$ in
$L_{\lambda_+}$ and a negative part $u_-$ in $L_{\lambda_-}:$
    $
  K_\lambda=\{u \hbox{ measurable};\exists a>0,
  \IZ\lambda(au)\,dR<\infty\}.
    $

\par\smallskip\noindent \textsf{The cone $\YYt$.}\index{YYt@$\YYt$}\
The $\sigma(K_{\lambda},L_\pm)$-closure $\overline{A}$ of a set
$A\subset K_\lambda$ is defined as follows: $u\in K_\lambda$ is in
$\overline{A}$ if ${u_\pm}$ is in the
$\sigma(L_{\lambda_\pm},L_{\lambda^*_\pm})$-closure of
$A_\pm=\{u_\pm;u\in A\}.$ Clearly,
$\overline{A}_\pm=\{u_\pm;u\in\overline{A}\}.$ The cone
$\YYt\subset\XX^*$ is the extension of $\YYo$ which is isomorphic
to the $\sigma(K_{\lambda},L_\pm)$-closure $\widetilde{T^*\YYo}$
of $T^*\YYo$ in $K_{\lambda}$ in the sense that
$T^*\YYt=\widetilde{T^*\YYo}.$

\par\smallskip\noindent\textsf{The extended dual problem ($\dtc$).}\
The extended dual problem associated with (\ref{pc}) is
\begin{equation*}
    \textsl{maximize } \inf_{x\in \CX}\langle\omega,x\rangle-I_\gamma(\langle\omega,\theta\rangle),
   \quad \omega\in\YYt\tag{$\dtc$}
   \index{Dual problems!D2@\eqref{dtc}}
\end{equation*}

Note that the dual bracket $\langle\omega,x\rangle$ is meaningful
for each $\omega\in\YYt$ and $x\in\XX.$

\subsection*{Some definitions needed for stating Theorem \ref{res-03}}

Again, one needs to introduce several cones to state the extended
dual problem $\dbc.$

Recall that there is a natural order on the algebraic dual
 space $E^*$ of a Riesz vector space $E$ which is defined by:
 $e^*\leq f^*\Leftrightarrow\langle e^*,e\rangle \leq \langle
 f^*,e\rangle $ for any $e\in E$ with $e\geq 0.$ A linear form
 $e^*\in E^*$ is said to be \emph{relatively bounded} if
 for any $f\in E,$ $f\geq 0,$ we have $\sup_{e:\vert e\vert
 \leq f} \vert \langle e^*,e\rangle \vert <+\infty .$ Although
 $E^*$ may not be a Riesz space in general, the vector space
 $E^b$ of all the relatively bounded linear forms on $E$ is always
 a Riesz space. In particular, the elements of $E^b$ admit a
 decomposition in positive and negative parts
 $e^*=e^*_+-e^*_-.$

\par\smallskip\noindent\textsf{The cone $K_\lambda''.$}\index{Klambda2@$K_\lambda''$}\
It is the cone of all relatively bounded linear forms
  $\zeta\in\Ll'^b$ on $\Ll'$ with a positive part $\zeta_+$ whose
  restriction to $L_{\lambda_+}'\subset\Ll'$ is in
  $L_{\lambda_+}''$ and with a negative part $\zeta_-$ whose
  restriction to $L_{\lambda_-}'\subset\Ll'$ is in
  $L_{\lambda_-}'':$
    $
  K_\lambda''=\{\zeta\in\Ll'^b;{\zeta_\pm}_{|L_{\lambda_\pm}'}\in
  L_{\lambda_\pm}''\}.
    $
Note that $L_{\lambda_\pm}'\subset\Ll'.$

\par\smallskip\noindent\textsf{A decomposition in $K_\lambda''.$}\
Let $\rho$ be any Young function.  Translating  decomposition
\eqref{eq-dec} onto $K_\lambda''$ leads to
$K_\lambda''=[K_\lambda\oplus K_{\lambda^*}^s]\oplus
K_\lambda^{s\prime}$ where one defines
 $
  K_{\lambda^*}^s=\{\zeta\in(\LlsR)^b;{\zeta_\pm}_{|L_{\lambda^*_\pm} R}\in
  L_{\lambda^*_\pm}^s\}
  $
  and
  $
  K_\lambda^{s\prime}=\{\zeta\in\Ll^{sb};{\zeta_\pm}_{|L_{\lambda_\pm}^s}\in
  L_{\lambda_\pm}^{s\prime}\}.
  $
Note that $L_{\lambda^*_\pm} R\subset\LlsR$ and
$L_{\lambda_\pm}^s\subset\Ll^s.$ With these cones in hand, the
decomposition \eqref{eq-dec} holds for any $\zeta\in K_\lambda''$
with
\begin{equation*}
    \left\{
\begin{array}{l}
  \zeta_1=\zeta_1^a+\zeta_1^s\in K_\lambda\oplus K_{\lambda^*}^s=K_{\lambda^*}',\\
   \zeta_2\in K_\lambda^{s\prime}. \\
\end{array}%
\right.
\end{equation*}

\par\smallskip\noindent\textsf{The set $\YYb.$}\index{YYb@$\YYb$}\
The $\sigma(K_{\lambda}'',L_{\pm}')$-closure $\overline{A}$ of a
set $A\subset K_\lambda''$ is defined as follows: $\zeta\in
K_\lambda''$ is in $\overline{A}$ if ${\zeta_\pm}$ is in the
$\sigma(L_{\lambda_\pm}'',L_{\lambda_\pm}')$-closure of
$A_\pm=\{\zeta_\pm;\zeta\in A\}.$ Clearly,
$\overline{A}_\pm=\{\zeta_\pm;\zeta\in\overline{A}\}.$ Let
$\overline{T^*\YYo}$ denote the
$\sigma(K_{\lambda}'',L_{\pm}')$-closure of $T^*\YYo$ in
$K_{\lambda}''.$
\\
Let $D$ denote the $\sigma(K_\lambda^{s\prime},L_\pm^s)$-closure
of $\dom I_{\lambda},$ that is $\zeta\in K_\lambda^{s\prime}$ is
in $D$ if and only if $\zeta_\pm$ is in the
$\sigma(L_{\lambda_\pm}^{s\prime},L_{\lambda_\pm}^s)$-closure of
$\{u_\pm;u\in \dom I_{\lambda}\}.$
\\
The set $\YYb\subset\XX^*$ is the extension of $\YYo$ which is
isomorphic to $\overline{T^*\YYo}\cap\{\zeta\in
K_{\lambda}'';\zeta_1^s=0, \zeta_2\in D\}$ in the sense that
$$T^*\YYb=\overline{T^*\YYo}\cap\{\zeta\in
K_{\lambda}'';\zeta_1^s=0,\zeta_2\in D\}.$$

\par\smallskip\noindent\textsf{The extended dual problem ($\dbc$).}\
 The extended dual problem associated with
(\ref{pbc}) is
\begin{equation*}
   \textsl{maximize } \inf_{x\in \CX}\langle\omega,x\rangle-
   I_\gamma\big([T^*\omega]_1^a\big),
   \quad \omega\in\YYb \tag{$\dbc$}
   \index{Dual problems!D3@\eqref{dbc}}
\end{equation*}

The exact statement of Theorem \ref{res-03}-(c) is

\begin{theorem}\label{res-03bis}
Under the assumptions of Theorem \ref{res-03}, we have
\begin{enumerate}[]
 \item[(c)]
 Let us denote $\xh:= T\lh.$ There exists  $\ob\in\YYb\index{o2@$\ob,$ see \eqref{eq-40}, \eqref{eq-40bis}}$ such that
 \begin{equation}\label{eq-40bis}
    \left\{\begin{array}{cl}
      (a) & \xh\in C\cap\dom\Gamma^* \\
      (b) & \langle \ob,\xh\rangle_{\XX^*,\XX} \leq \langle \ob,x\rangle_{\XX^*,\XX}, \forall x\in C\cap\dom\Gamma^* \\
      (c) & \lh\in\gamma'_z([T^*\ob]_1^a)\,R+D^\bot([T^*\ob]_2) \\
    \end{array}\right.
\end{equation}
where
$$
D^\bot(\eta)=\{k\in \Ll^s; \forall h\in \Ll, \eta+h\in D
\Rightarrow \langle h,k\rangle\le0\}
$$
is the outer normal cone of $D$ at $\eta.$
\\
$T^*\ob$ is in the $\sigma(K_\lambda'',L_\pm')$-closure of
    $T^*(\{y\in\YYo; \IZ\lambda(\langle y,\theta\rangle)\,dR<\infty\})$
and there exists some $\ot\in\XXo^*$ such that
$$
[T^*\ob]_1^a=\langle\ot,\theta(\cdot)\rangle_{\XXo^*,\XXo}
$$
is a measurable function in the strong closure of $T^*(\{y\in\YYo;
\IZ\lambda(\langle y,\theta\rangle)\,dR<\infty\})$ in $K_\lambda:$
the set of all $u\in K_\lambda$ such that $u_\pm$ is in the strong
closure of $T^*(\{y\in\YYo; \IZ\lambda(\langle
y,\theta\rangle)\,dR<\infty\})_\pm$ in $L_{\lambda_\pm}.$
\\
Furthermore, $\lh\in\Ll'$ and $\ob\in\YYb$ satisfy
\eqref{eq-40bis} if and only if $\lh$ solves \eqref{pbc} and $\ob$
solves \eqref{dbc}.
\end{enumerate}
\end{theorem}

The dual bracket $\langle [T^*\ob]_2,\lh^s \rangle$ in Theorem
\ref{res-03}-(d.3) is intended to be
 $\langle [T^*\ob]_2,\lh^s \rangle_{{K_\lambda^s}',K_\lambda^s}.$

\subsection*{A complement to Theorem \ref{res-07}}
As an easy corollary of Theorem \ref{res-03bis} and the proof of
Theorem \ref{res-07} we obtain the following

\begin{proposition}[Complement to Theorem \ref{res-07}]\label{res-07bis}
For each  $\od\index{od@$\od,$ see \eqref{eq-18}}$ satisfying
\eqref{eq-18} at Theorem \ref{res-07}, $\langle\od,
\theta(\cdot)\rangle$ is in the strong closure of $T^*\dom\Gamma$
in $K_\lambda:$ the set of all $u\in K_\lambda$ such that $u_\pm$
is in the $\|\cdot\|_{\lambda_\pm}$-closure of
$\{\yt_\pm;y\in\dom\Gamma\}$ in $L_{\lambda_\pm},$ and there
exists $\ob\in\YYb$ solution of \emph{(\ref{dbc})} such that
$[T^*\ob]_1^a=\langle\od, \theta(\cdot)\rangle.$
\end{proposition}

Recall that $\Gamma$ is defined at \eqref{eq-45}.

\section{Proof of Theorem \ref{res-09}}\label{sec:proofcondLLN}

This proof is an easy variation on \cite[Section 7.3]{DZ}. It is
given for the reader's convenience. Theorem \ref{res-09} is a
restatement of Propositions \ref{uppercond}, \ref{lowercond},
\ref{condLLN-X} and \ref{condLLN-L} below.
\\
We begin with the proof for $X_n.$ As $T$ is continuous and $I$ is
a good rate function, $J$ is also a good rate function. Let us
first state the upper bound of a conditional LDP.

\begin{proposition}
\label{uppercond} Under the assumptions of Theorem \ref{res-09},
for all closed subset $F$ of $\XXo,$ we have
\begin{displaymath}
  \limsupdn \frac 1n \log\mathbb{P}(X_n\in F\mid X_n\in C_\delta)\leq -J_{C}(F)
\end{displaymath}
where
\begin{displaymath}
  J_{C}(x):= \left\{
    \begin{array}{ll}
J(x)-J(C) &\mathrm{if\ }x\in C\\
+\infty &\mathrm{if\ }x\not\in C
    \end{array}
\right.
\end{displaymath}
\end{proposition}
\proof Clearly, for all measurable set $B$ and all $n\geq 1,$
$\delta>0,$ we have $\frac 1n \log\mathbb{P}(X_n\in B\mid X_n\in
C_\delta)=\frac 1n \log\mathbb{P}(X_n\in B\cap C_\delta)-\frac
1n\log\mathbb{P}(X_n\in C_\delta).$ Hence,
\begin{eqnarray*}
&&  \limsup_n \frac 1n\log\mathbb{P}(X_n\in F\mid  X_n\in C_\delta)\\
&\leq& \limsup_n\frac 1n \log\mathbb{P}(X_n\in F\cap \cl C_\delta)
-\liminf_n\frac 1n\log\mathbb{P}(X_n\in \inter C_\delta)\\
&\leq& -J(F\cap C_\delta)+J(\inter C_\delta)\\
&\leq& -J(F\cap C_\delta)+J(C)
\end{eqnarray*}
where the last inequality follows from the assumption (B$_C^3$).
We complete the proof with the following lemma.
\endproof

\begin{lemma}
\label{lem-J} For any closed set $F,$ $\lim_\delta J(F\cap
C_\delta)=J(F\cap C)\in[0,\infty].$
\end{lemma}
\proof As $C\subset C_\delta,$ for all $\delta>0,$ we have
$J(F\cap C_\delta)\leq J(F\cap C).$ Since $C_\delta$ is
decreasing, $J(F\cap  C_\delta)$ is nondecreasing and $\lim_\delta
J(F\cap  C_\delta)=\sup_\delta J(F\cap C_\delta)\in[0,\infty].$ If
$\sup_\delta J(F\cap C_\delta)=\infty,$ the inequality $J(F\cap
C_\delta)\leq J(F\cap C)$ leads to the desired result.
\\
Now, let us suppose that $\sup_\delta J(F\cap C_\delta)<\infty.$
As $F\cap C_\delta$ is closed and $J$ is inf-compact, for any
$\delta$ there exists $x_\delta\in F\cap C_\delta$ such that
$J(x_\delta)=J(F\cap C_\delta)$ and we can extract a converging
subsequence $x_k \rightarrow x_*.$ Because the $C_\delta$'s are
decreasing, we get $\cap_\delta C_\delta =\cap_k C_{\delta_k}$ and
$\lim_\delta J(F\cap C_\delta)=\lim_k J(x_k).$ More, $x_*\in F\cap
(\cap_k C_{\delta_k})=F\cap C$ and as $J$ is lsc: $\lim_k
J(x_k)\geq J(x_*)\geq J(F\cap C).$ Therefore, $\lim_\delta J(F\cap
C_\delta)\geq J(F\cap C)$ which completes the proof.
\endproof
Let us state the lower bound corresponding to Proposition
\ref{uppercond}.

\begin{proposition}
\label{lowercond} If the assumption \emph{(B$_C^3$)} on the
conditioning event is restricted to (3-b):
\begin{equation}
  \label{eq:cond2}
  C\subset \inter C_\delta, \forall\delta>0
\end{equation}
then, for all open subset $G$ of $\XXo,$ we have
\begin{displaymath}
 \liminfdn \frac 1n\log\mathbb{P}(X_n\in G\mid X_n\in C_\delta)\geq -J_{C}(G).
\end{displaymath}
\end{proposition}
\proof For all $\delta>0,$
\begin{eqnarray*}
&&  \liminf_n\frac 1n\log\mathbb{P}(X_n\in G\mid X_n\in C_\delta)\\
&\geq& \liminf_n\frac 1n\log\mathbb{P}(X_n\in G\cap\inter
C_\delta)
-\limsup_n \frac 1n\log\mathbb{P}(X_n\in C_\delta)\\
&\geq& -J(G\cap \inter C_\delta)+J( C_\delta)\\
&\geq& -J(G\cap C)+J( C_\delta).
\end{eqnarray*}
We conclude with Lemma \ref{lem-J}.
\endproof

Let us recall that $\mathcal{H}=\mathrm{argmin} J_{C}$ is the set
of the minimizers of $J$ on $C.$  As $C$ is closed and $J$ is
inf-compact, $\mathcal{H}$ is a compact set. As an immediate
corollary of  Proposition \ref{uppercond}, we have the following
conditional LLN which is the part of the statement of Theorem
\ref{res-09} concerning $X_n.$
\begin{proposition}
\label{condLLN-X} For all open subset $H$ of $\XXo$ such that
$\mathcal{H}\subset H,$ we have
\begin{displaymath}
\limsupdn \frac 1n \log\mathbb{P}(X_n\not\in H\mid X_n\in
C_\delta)<0.
\end{displaymath}
\end{proposition}

Let us now have a look at $L_n.$ We are interested in the
asymptotic behavior of $\mathbb{P}(L_n\in\cdot\mid TL_n\in
C_\delta)$ with $C_\delta\subset \XXo.$ Let us denote $A_\delta:=
T^{-1}(C_\delta)=\{\ell\in\LLo ; T\ell\in C_\delta\}$ and $A=
T^{-1}C.$ It is useful to state the assumptions on the
$C_\delta$'s rather than on the $A_\delta$'s. In fact, we have the
following transfer result.

\begin{lemma}
\label{condlem} We assume that $T$ is continuous.
\begin{itemize}
\item[(a)] If $C$ is closed and $J(C)=J(\inter C),$ then $A$ is
closed and $I(A)=I(\inter A).$

\item[(b)] If $C:=\cap_{\delta} C_\delta\subset \inter C_\delta$
for all $\delta>0,$ then $A=\cap_{\delta}\cl A_\delta$ and
$A\subset\inter A_\delta$ for all $\delta>0.$
\end{itemize}
\end{lemma}
\proof Since $T$ is continuous, for any $A'=T^{-1}C',$ we have:
$T^{-1}(\inter C')\subset\inter A'\subset A'\subset \cl A'\subset
T^{-1}(\cl C').$
\\
Let us begin with (a).  As $C$ is closed, so is $A.$ For any
$A=T^{-1}C,$ we have $I(A)=\inf\{\Phi^*(\ell) ; T\ell\in C\}=
\inf_{x\in C} \inf\{\Phi^*(\ell); T\ell=x\}=\inf_{x\in C}
J(x)=J(C).$ Hence, $I(A)=J(C)=J(\inter C)$ (by hypothesis)
$=I(T^{-1}(\inter C)\geq I(\inter A),$ since $T^{-1}(\inter
C)\subset \inter A.$ But the converse inequality: $I(A)\leq
I(\inter A)$ is clear.
\\
Let us prove (b). We have: $\cap_\delta \cl A_\delta \subset
T^{-1}(\cap_\delta  C_\delta):= A:= T^{-1}(C)\subset
T^{-1}(\cap_\delta \inter C_\delta)$ (by hypothesis) $\subset
\cap_\delta \inter A_\delta \subset \cap_\delta \cl A_\delta.$
This proves that all these sets are equal, and in particular that
$A=\cap_\delta \cl A_\delta.$ On the other hand, as for any
$\delta>0,$ $C\subset \inter C_\delta,$ we have
$A=T^{-1}(C)\subset T^{-1}(\inter C_\delta)\subset \inter
A_\delta.$
\endproof

Let us recall that $\mathcal{G}$ is the set of the minimizers of
$I$ on $A.$ By the above Lemma \ref{condlem}, in the situation of
the $L_n$'s, Proposition \ref{condLLN-X} becomes the following

\begin{proposition}
\label{condLLN-L} Under our general assumptions, for all open
subset $G$ of $\LLo$ such that $\mathcal{G}\subset G,$ we have
\begin{displaymath}
\limsupdn \frac 1n \log\mathbb{P}(L_n\not\in G\mid TL_n\in
C_\delta)<0.
\end{displaymath}
\end{proposition}
Note that by Lemma \ref{condlem}, the $A_\delta$'s  share the same
properties as the $C_\delta$'s. In particular, $I(\inter
A_\delta)\leq I(A)$ also holds for them.


\begin{theindex}

  \item Assumptions (A), 7
    \subitem (A$_R$) on the measure $R$, 7
    \subitem (A$_C$) on the set $C$, 7
    \subitem (A$_{\gamma^*}$) on $\gamma^*$, 8
    \subitem (A$_\theta$) on $\theta$, 8
      \subsubitem \eqref{A-forall}, 9
      \subsubitem \eqref{A-exists}, 9
  \item Assumptions (B), 30
    \subitem (B$_L$) on $\{L_n\}$, 30
    \subitem (B$_T$) on $T$, 30
    \subitem (B$_C$) on the constraint set $C$, 30

  \indexspace

  \item $B_\ZZ,$ bounded measurable functions on $\ZZ$, 31

  \indexspace

  \item $\CC,$ constraint set, 2, 4, 16
  \item $C,$ constraint set, 4, 16
  \item $\CCb,$ extended constraint set, 16
  \item convex conjugate, $f^*,$ see \eqref{eq-61}, 6
  \item $C_\ZZ,$ bounded continuous functions on $\ZZ$, 32

  \indexspace

  \item dominating point, see Definition \ref{defdomp}, 27
  \item Dual problems
    \subitem \eqref{dc}, 10
    \subitem \eqref{dtc}, 10, 33
    \subitem \eqref{dbc}, 12, 34
    \subitem \eqref{eq-58}, 19
    \subitem \eqref{eq-59}, 20

  \indexspace

  \item $\Eexp$, 31
  \item entropic projection, see Definition \ref{defCsiszar}, 23
  \item $E_\tau$, 14

  \indexspace

  \item $\Phi_E, \Phi_L$, 18
  \item $\Phi_E^*, \Phi_L^*$, 18
  \item Functions of $(s,z)$ or $(t,z)$
    \subitem $\gamma,$ see \eqref{eq-49}, 8
    \subitem $\gamma^*,$ integrand, see \eqref{eq-25}, 2, 3, 7
    \subitem $\lambda,$ see \eqref{eq-49}, 8
    \subitem $\lmax,$ see \eqref{eq-50}, 8
    \subitem $m,$ see (A$_{\gamma^*}$), 8
    \subitem $\rho,$ Young function, 6
    \subitem $\tau$, 14
    \subitem $\tau^*$, 14

  \indexspace

  \item $\mathcal{G}$, 31
  \item $\Gamma,$ see \eqref{eq-45}, 10
  \item $\Gamma^\ast,$ see \eqref{eq-53}, 10
  \item $\widetilde{\Gamma^*},$ see \eqref{eq-60}, 28
  \item generalized entropic projection, see Definition \ref{defCsiszar},
        23
  \item generalized $I$-projection, see Definition \ref{defCsiszar}, 23

  \indexspace

  \item $\mathcal{H}$, 31

  \indexspace

  \item $\iota_A,$ convex indicator, see \eqref{eq-42}, 3
  \item $I_\varphi,$ integral functional, 6
  \item $I,$ entropy, see \eqref{eq-16}, \eqref{eq-08}, 2, 9, 18
  \item $\Ib,$ extended entropy, see \eqref{III}, \eqref{eq-08}, 9, 18
  \item $I^s,$ singular entropy, see \eqref{eq-51}, 9
  \item $I(\cdot\mid R),$ relative entropy, see \eqref{eq-54}, 2, 14
  \item $\Ib(\cdot\mid R),$ extended relative entropy, see \eqref{ext-entrop},
        14
  \item $\icor$, 6
  \item $\icordom$, 6
  \item inf-convolution, $f\Box g$, 6
  \item $I$-projection, see Definition \ref{defCsiszar}, 23

  \indexspace

  \item $J$, 27, 30
  \item $\Jb,J^s$, 27

  \indexspace

  \item $K_\lambda$, 33
  \item $K_\lambda''$, 34

  \indexspace

  \item $\Lexp$, 31
  \item $\lh,$ solution to \eqref{pbc}, see \eqref{eq-40}, 12
  \item $\Ll$, 8
  \item $\LLo$, 30
  \item $\Lls$, 8
  \item $L_n$, 1, 2, 30
  \item $L_\tau$, 14
  \item $L_{\tau^*}$, 14

  \indexspace

  \item $\MZ,$ space of signed measures on $\ZZ$, 2

  \indexspace

  \item non-recessive $x,$ see Definition \ref{def-01}, 28

  \indexspace

  \item $\ot,$ see \eqref{eq-36}, 11
  \item $\ob,$ see \eqref{eq-40}, \eqref{eq-40bis}, 12, 34
  \item $\od,$ see \eqref{eq-18}, 15, 35
  \item $\OZ,$ see \eqref{eq-44}, 14
  \item Orlicz spaces, 6
    \subitem decomposition
      \subsubitem $\la,$ absolutely continuous part of $\ell,$ see \eqref{decomp},
        7
      \subsubitem $\lsing,$ singular part of $\ell,$ see \eqref{decomp},
        7
      \subsubitem $\zeta_1^a,\zeta_1^s,\zeta_2,$ see \eqref{eq-dec}, 12
      \subsubitem $\Lrs R,$ space of absolutely continuous forms, 7
      \subsubitem $\Lrsi,$ space of singular forms, 7
    \subitem $\Delta_2$-condition, see \eqref{delta2}, 7
    \subitem $\Er,$ small Orlicz space, 6
    \subitem $\Lr,$ Orlicz space, 6
    \subitem $\Vert\cdot\Vert_ \rho,$ Luxemburg norm, 6
    \subitem $\rho,$ Young function, 6

  \indexspace

  \item $P_\diamond,$ see \eqref{eq-34}, 25
  \item $P_{\mathrm{exp}}$, 31
  \item $\widehat{P},$ see \eqref{eq-33}, 24
  \item Primal problems
    \subitem \eqref{pc}, 9
    \subitem \eqref{pbc}, 9
    \subitem \eqref{pcc}, 16
    \subitem \eqref{eq-55}, 19
    \subitem \eqref{eq-56}, 19
    \subitem \eqref{eq-57}, 19
  \item $\PZ,$ set of probability measures on $\ZZ$, 2

  \indexspace

  \item $\Qd,$ see \eqref{eq-18}, 15
  \item $\Qh,$ minimizer of \eqref{pc}, 11
  \item $Q_*,$ generalized  $I$-projection, 23

  \indexspace

  \item $R,$ reference measure, 2, 7
  \item recessive $x,$ see Definition \ref{def-01}, 28
  \item reverse relative entropy, 3

  \indexspace

  \item $T,$ constraint operator, see \eqref{eq-52}, 9
  \item $\theta,$ constraint function, see \eqref{eq-52}, 4, 7, 9
  \item $T^\ast,$ adjoint of $T$, 10

  \indexspace

  \item $\UUo$, 30

  \indexspace

  \item $\xh,$ see \eqref{eq-36}, \eqref{eq-40}, 11, 12
  \item $\Xi,$ see \eqref{cramer}, 24
  \item $\Xi$-dominating point, see Definition \ref{dompointNey}, 27
  \item $\XX,$ topological dual of $\YYo$ and $\YY$, 10
  \item $\XXo,$ algebraic dual of $\YYo$, 7
  \item $\xb^a, \xb^s$, 16
  \item $x^a, x^s$, 27

  \indexspace

  \item $\YY,$ completion of $\YYo$, 10
  \item $\YYo$, 7
  \item $\YYb$, 12, 34
  \item $\YYt$, 10, 33

  \indexspace

  \item $\ZZ,$ reference space, 7

\end{theindex}


\end{document}